\documentclass[11pt,a4paper]{article}

\usepackage[utf8x]{inputenc}
\usepackage[T1]{fontenc}
\usepackage[pdftex]{graphicx}
\usepackage[pdftex,linkcolor=black,pdfborder={0 0 0}]{hyperref}
\usepackage{xcolor}
\usepackage{calc}
\usepackage{amsmath,amsfonts,amssymb,amsthm,mathtools}
\usepackage[a4paper, lmargin=0.125\paperwidth, rmargin=0.125\paperwidth, tmargin=0.1111\paperheight, bmargin=0.1111\paperheight]{geometry}
\usepackage{natbib}
\usepackage[all]{nowidow}
\usepackage[protrusion=true,expansion=true]{microtype}
\usepackage{subcaption}
\usepackage{tabularx}
\usepackage{float}
\usepackage{times}
\usepackage{setspace}
\usepackage{enumerate}
\usepackage{booktabs}
\usepackage{multirow}
\usepackage{multicol}
\usepackage{adjustbox}
\usepackage{algorithm}
\usepackage{algpseudocode}
\usepackage{placeins}
\title{Curve Band Depth: A Band-Based Data Depth for Unparameterized Planar Curves}
\author{Siyi Wang, Alexandre Leblanc, Paul D. McNicholas}
\date{}

\newtheorem{definition}{Definition}[section]
\newtheorem{theorem}{Theorem}[section]
\newtheorem{proposition}{Proposition}[section]
\newtheorem{lemma}{Lemma}[section]

\theoremstyle{remark}
\newtheorem{remark}{Remark}[section]

\newcommand{\bbS}{\mathbb{S}}

\newcommand{\R}{\mathbb{R}}
\newcommand{\Cset}{\mathcal{C}}

\newcommand{\Borel}{\mathcal{B}}
\newcommand{\Rband}{\mathcal{R}}
\newcommand{\one}{\mathbf{1}}
\newcommand{\dist}{\operatorname{dist}}

\newcommand{\cl}{\operatorname{cl}}
\newcommand{\arc}{\mathrm{arc}}

\begin{document}
\maketitle

\begin{abstract}
We introduce \emph{curve band depth} (CBD), a band-based data depth for samples of \emph{unparameterized} planar curves.
CBD is motivated by band depth and modified band depth for functional data, but targets trajectory data. Unlike the halfspace-based curve depth of \citet{de2021depth} and the curve stabbing depth of \citet{durocher2023csd}, CBD is defined through a geometric band region generated by two curves, and measures the arc-length proportion of a target curve lying inside such bands.
We develop a CBD family consisting of an integral version (int-CBD), an infimal version (inf-CBD), and a fast-walk variant (FW-CBD). The fast-walk band is a narrower band construction contained in the global convex-combination band.
We establish boundedness, vanishing at infinity, and similarity invariance for these constructions, together with a Borel-measurability result for the induced depth maps under a mild measurability assumption.
A length-penalized variant is proposed for samples with heterogeneous curve lengths.
We implement the methods via arc-length sampling and polygonal approximations, and evaluate them through classification of overlapping handwriting data and MNIST-derived digit curves, online-signature screening on \texttt{MOBISIG}, and an exploratory clustering task based on decomposed band contributions.
\end{abstract}

\section{Introduction}
Data depth provides a center-outward ordering of observations and a natural
framework for robust inference~\citep{mosler2013depth}. In multivariate
analysis, the literature begins with Tukey's halfspace depth~\citep{Tukey1975} and was soon
broadened by several geometric alternatives, including Oja depth~\citep{oja1983descriptive}, simplicial
depth~\citep{liu1990notion}, projection depth~\citep{donoho1992breakdown}, the $L_1$-depth~\citep{vardi2000multivariate}, and spatial depth~\citep{serfling2002depth}; general notions of statistical depth are formalized in \citet{zuo2000general}.

These ideas were subsequently extended beyond point clouds to more complex
objects. For functional data, a large literature now exists, including
integral-type constructions such as the integrated depth of
\citet{nagy2016integrated} and the modified band depth of
\citet{lopez2009concept}, as well as infimal-type constructions, such as $\Phi$-depth~\citet{mosler2013depth}; see also
\citet{cuevas2007robust,mosler2012general}.

In \emph{curve data} problems, observations are planar or spatial trajectories, and parameterization is typically an artifact: a trajectory can be traced at varying speeds, in either direction, or (if closed) from any starting point without changing its essential geometric character.
Recent work has therefore developed depth notions for unparameterized curves.
The \emph{curve depth} of \citet{de2021depth} and the \emph{curve stabbing depth} of \citet{durocher2023csd} extend halfspace depth to unparameterized curves, in both cases by integrating a local centrality score along arc length (we give precise definitions in Section~\ref{sec:related_work}).
In the functional-depth taxonomy both constructions belong to the integral family.

This paper takes a different route: a \emph{band-based} depth for unparameterized curves, inspired by band depth (BD) and modified band depth (MBD) for functional data.
Two structural reasons motivate the new construction.
First, functional BD and MBD compare scalar values at a common ordered domain index, so they require a shared parameterization across the sample; for unparameterized planar curves there is no such shared index because the curves are identified up to monotone reparameterization, and so functional BD does not transfer directly.
Second, the curve depths of \citet{de2021depth,durocher2023csd} extend halfspace and ray notions of centrality: they answer ``how separable is this curve from the rest by a halfspace?'' rather than ``how much of this curve lies inside the geometric region between two other curves?''.
The latter is a band-occupancy notion, and is what BD and MBD capture in the functional setting.
We therefore define a geometric ``band region'' between two curves as a parameterization-free subset of the plane, and measure how much of a target curve lies in that band along arc length.
The resulting family of depths admits both an integral form (averaging in-band proportions) and an infimal form (requiring full inclusion), parallel to the MBD/BD distinction.

\subsection{Contributions}
Our contributions are as follows.
\begin{enumerate}[(1)]
  \item We define \emph{integral CBD} (int-CBD) and \emph{infimal CBD} (inf-CBD), and formalize a \emph{CBD-family} template that covers both.
  \item We propose \emph{fast-walk CBD} (FW-CBD), a CBD-family member whose band is the planar region enclosed by a closed walk along the two generating curves. 
  \item We prove boundedness, vanishing at infinity, similarity invariance, and conditional Borel-measurability for CBD and FW-CBD.
  \item We describe an implementation of CBD and FW-CBD and report experiments on overlapping handwriting (\texttt{mrfDepth} ``a'' vs.\ ``i''), MNIST-derived digit curves (1 vs.\ 7), and the \texttt{MOBISIG} online-signature-screening task. We also briefly examine their potential use for clustering by reusing decomposed band contributions on character trajectories.
\end{enumerate}

The rest of the paper is organized as follows.
Section~\ref{sec:related_work} reviews functional and curve depth background and introduces the curve space and arc-length measure used in subsequent sections.
Section~\ref{sec:cbd_def} defines the proposed band constructions and the CBD family.
Section~\ref{sec:properties} establishes their main theoretical properties.
Section~\ref{sec:implementation} discusses implementation details of proposed depths.
Section~\ref{sec:experiments} reports numerical illustrations and empirical results.
Section~\ref{sec:discussion} concludes with a brief discussion of limitations and directions for future work.

\section{Background and related work}
\label{sec:related_work}
\subsection{Functional depths: integral vs infimal}
Functional depths fall into two groups. \emph{Integral} constructions integrate a pointwise depth along the domain, for example the integrated depth of \citet{nagy2016integrated} and the modified band depth (MBD) of \citet{lopez2009concept}. \emph{Infimal} constructions instead impose a global inclusion criterion.
Band depth (BD) counts the proportion of bands formed by subsets of sample functions that contain the entire function; MBD replaces full inclusion by the proportion of the domain where the function lies within the band.
This integral/infimal split motivates the two curve-band depths we introduce.
Related multivariate-functional constructions include the simplicial band depth of \citet{lopez2014simplicial}, the multivariate functional halfspace depth of \citet{claeskens2014multivariate}, the half-region depth of \citet{lopezpintado2011halfregion}, and the fast exact BD/MBD algorithms of \citet{sun2012exact}. General treatments are given by \citep{mosler2012general,nagy2016integrated},
with \citet{nagy2017depth} addressing depth-based detection of shape outlying
functions and \citet{nagy2024depth} examining the choice of depth for
constructing functional boxplots.
\subsection{Depth for curve data}
The curve depth of \citet{de2021depth} extends Tukey's halfspace depth~\citet{Tukey1975} to unparameterized curves by comparing (i) the probability that a random curve intersects a halfspace with (ii) the arc-length proportion of the target curve lying in that halfspace, and integrating the result along the curve.
Curve stabbing depth \citep{durocher2023csd} defines depth via stabbing counts of rays intersecting sample curves, and integrates along the target curve.
Both methods are parameterization-invariant, but they are halfspace-based and do not produce a band-region depth in the ambient plane.

Band-based notions for unparameterized curves are less developed. CBD defines bands directly as subsets of the ambient plane through the convex-combination construction of Section~\ref{sec:cbd_def}, and weights inclusion by arc length. This transfers the BD/MBD idea to unparameterized planar curves.
\subsection{Curve space and arc-length measure}
\label{sec:curve_space}
We work with parameterized rectifiable planar curves modulo monotone reparameterization.

\paragraph{Parameterized rectifiable curves.}
For an interval domain $I\in\{[0,1],\bbS^1\}$ where $\bbS^1=[0,1]/\{0\sim 1\}$, let $\mathcal F(I)$ denote the set of all continuous maps $f:I\to\R^2$ whose total variation
\[
L(f)=\sup_{0=\tau_0<\cdots<\tau_J=1}\sum_{j=1}^J \|f(\tau_j)-f(\tau_{j-1})\|_2
\]
is finite (rectifiability) and whose arc-length map $s_f(t):=L(f|_{[0,t]})$ is strictly increasing (no stationary intervals); for the closed-domain case $I=\bbS^1$ we identify $f$ with any continuous lift to $[0,1]$ satisfying $f(0)=f(1)$ and the same strict-increase condition.
The strict-increase assumption excludes parameterizations that are constant on a positive-length sub-interval and ensures every $f$ admits a continuous arc-length reparameterization \citep[Theorem 2.7.6]{burago2001course}.
We call any such $f$ a \emph{parameterized rectifiable curve}; its \emph{trace} is the image set $\Gamma_f=f(I)\subseteq\R^2$.
The choice of domain encodes whether the curve is open ($I=[0,1]$) or closed ($I=\bbS^1$).

\paragraph{Equivalence by monotone reparameterization.}
For $f_1,f_2\in\mathcal F(I)$, write $f_1\sim f_2$ if there exists an orientation-preserving or orientation-reversing homeomorphism $\phi:I\to I$ with $f_1=f_2\circ\phi$.
On $I=[0,1]$ the homeomorphism is required to satisfy $\phi(0)=0,\phi(1)=1$ (monotone increasing) or $\phi(0)=1,\phi(1)=0$ (monotone decreasing).
On $I=\bbS^1$ it is required to be either an orientation-preserving homeomorphism (which includes all rotations $s\mapsto s+\theta\bmod 1$) or an orientation-reversing one.
The relation $\sim$ is an equivalence on $\mathcal F(I)$, and the \emph{curve space} is the disjoint union of quotients
\[
\Cset \;:=\; \bigl(\mathcal F([0,1])/\!\sim\bigr) \;\sqcup\; \bigl(\mathcal F(\bbS^1)/\!\sim\bigr).
\]
We write $C\in\Cset$ and refer to $f\in C$ as a \emph{representative} of $C$.

\begin{remark}[Properties intrinsic to $C\in\Cset$]
\label{rmk:monotone_reparam}
The equivalence $\sim$ identifies parameterizations that differ only by monotone time-warping, orientation reversal, and (for closed curves) choice of starting point, but distinguishes parameterizations that retrace portions of the trace.
Three consequences hold for every $C\in\Cset$:
(i) the length $L(C):=L(f)$ is independent of the representative $f\in C$ and equals the total variation of any representative;
(ii) the open/closed property (the choice of domain $I$) is intrinsic to $C$;
(iii) the trace $\Gamma_C=\Gamma_f\subseteq\R^2$ is intrinsic, but $\Cset$ is strictly finer than the quotient by trace-equivalence, e.g., a Y-shaped trace admits multiple non-equivalent parameterized curves in $\Cset$, distinguished by which arm is the start and which is the end.
\end{remark}

\paragraph{Arc-length parameterization and arc-length measure.}
For each $C\in\Cset$ with $L(C)>0$, there is a unique \emph{arc-length representative}
\[
\beta_C^{\arc} : [0,L(C)] \to \R^2 \;\;\text{(open case)}, \qquad
\beta_C^{\arc} : [0,L(C)]/\{0\sim L(C)\} \to \R^2 \;\;\text{(closed case)},
\]
satisfying $|(\beta_C^{\arc})'|=1$ almost everywhere, unique up to orientation reversal in the open case, and up to both circular shift and orientation reversal in the closed case.
Existence and uniqueness on the strict-increase class introduced above are standard: the arc-length map $s_f$ is a homeomorphism $[0,1]\to[0,L(f)]$, and $\beta_C^{\arc}=f\circ s_f^{-1}$ is well-defined on $[0,L(C)]$ for any representative $f$ of $C$, with the listed indeterminacy.
The arc-length probability measure on $\R^2$ is defined by
\begin{equation}\label{eq: arc-length}
    \mu_C(B):=\frac{1}{L(C)}\int_0^{L(C)} \one_B(\beta_C^{\arc}(s))\,ds, \qquad B\in\Borel(\R^2),
\end{equation}
where $\mathcal B(\R^2)$ denotes the Borel $\sigma$-algebra on $\R^2$ and the resulting measure depends only on $C$.
Indeed, suppose $\widetilde\beta_C^{\arc}$ is a second arc-length representative of $C$.
By uniqueness of arc-length parameterization (up to the listed indeterminacy), one of four cases holds:
(0)~$\widetilde\beta_C^{\arc}=\beta_C^{\arc}$ (identity; equality of integrals is immediate),
(i)~$\widetilde\beta_C^{\arc}(s)=\beta_C^{\arc}(L(C)-s)$ (orientation reversal, open or closed),
(ii)~$\widetilde\beta_C^{\arc}(s)=\beta_C^{\arc}(s+s_0\bmod L(C))$ for some $s_0\in[0,L(C))$ (circular shift, closed case),
(iii)~$\widetilde\beta_C^{\arc}(s)=\beta_C^{\arc}(s_0-s\bmod L(C))$ for some $s_0\in[0,L(C))$ (shifted reversal, closed case).
In case~(i), the substitution $u=L(C)-s$ gives $\int_0^{L(C)}\one_B(\widetilde\beta_C^{\arc}(s))\,ds = \int_0^{L(C)}\one_B(\beta_C^{\arc}(u))\,du$.
In case~(ii), periodicity of $\beta_C^{\arc}$ and invariance of Lebesgue measure under $s\mapsto s+s_0$ give the same equality.
In case~(iii), the substitution $u=s_0-s\bmod L(C)$ gives a measure-preserving map (Lebesgue measure on $[0,L(C))$ is invariant under $s\mapsto s_0-s\bmod L(C)$), so again $\int_0^{L(C)}\one_B(\widetilde\beta_C^{\arc}(s))\,ds = \int_0^{L(C)}\one_B(\beta_C^{\arc}(u))\,du$.
Hence $\mu_C$ is independent of which arc-length representative is used.

\paragraph{Identification convention for set operations.}
Whenever an expression involves set-theoretic operations on a curve $C\in\Cset$, such as $C\cap B$, $\dist(C,K)$, or $C\subseteq R$ for $R\subseteq\R^2$, we identify $C$ with its trace $\Gamma_C\subseteq\R^2$.
This is unambiguous because $\Gamma_C$ is intrinsic to $C$ (Remark~\ref{rmk:monotone_reparam}).

\paragraph{Borel structure on $\Cset$.}
\label{par:borel_C}
For measurability purposes we represent each $C\in\Cset$ by its length together with its normalized arc-length representative, understood modulo the intrinsic indeterminacies of arc-length parameterization, i.e., orientation reversal in the open case, and circular shift and reversal in the closed case.
Formally, let $\Phi$ send $C$ to its equivalence class $[L(C),\,\beta_C^{\arc}(L(C)\,\cdot)]$ in $(0,\infty)\times C(I,\R^2)/\!\sim_{\mathrm{arc}}$, where $\sim_{\mathrm{arc}}$ identifies representatives that differ by the above indeterminacies, and $C(I,\R^2)$ carries the supremum metric.
Equip $\Cset$ with the smallest $\sigma$-algebra making $\Phi$ measurable.
This makes $L(\cdot)$ and $\Gamma_{(\cdot)}$ (as a set-valued map into the hyperspace of compact subsets of $\R^2$ with the Hausdorff metric) Borel-measurable~\citep{molchanov2017theory}; any functional of $C$ that depends only on the trace and length, such as $\mu_C(B)$ for fixed Borel $B$, inherits measurability through $\Phi$.

\begin{remark}[Degenerate and non-atomic cases]\label{rmk:degenerate_nonatomic}
When $L(C)=0$ ($C$ is a degenerate point-curve), set $\mu_C$ to the Dirac measure at $\Gamma_C$.
The remainder of the paper restricts attention to $C\in\Cset$ with $L(C)>0$, and assumes $\mu_C$ is non-atomic (which holds whenever $\beta_C$ is not constant on a positive-Lebesgue-measure subset of $[0,L(C)]$).
\end{remark}

\subsection{Notation summary}
Throughout the paper, we use the following notations.
\begin{center}
\begin{tabular}{ll}
\hline
$\Cset$ & curve space \\
$L(C)$ & arc length of $C\in\Cset$ \\
$\Gamma_C$ & trace (image set) of $C$ in $\R^2$ \\
$\mu_C$ & arc-length probability measure of $C$ on $\R^2$ \\
$\beta_C^{\arc}$ & arc-length representative on $[0,L(C)]$ (or its circle analogue) \\
$\beta_C$ & normalized constant-speed representative on $[0,1]$ (or $\bbS^1$) \\
$\Rband(A,B)$ & global (convex-combination) band of $(A,B)$, Equation~\eqref{eq:global_band} \\
$R_{\mathrm{FW}}(A,B)$ & fast-walk band of $(A,B)$, Equation~\eqref{eq:fw-band} \\
$\widehat\Rband$ & generic band when statements apply to either construction \\
$D_{\mathrm{int}},D_{\mathrm{inf}}$ & integral and infimal CBD with the global band \\
$D_{\mathrm{FW}},D_{\mathrm{FW}}^{\inf}$ & integral and infimal CBD with the FW band \\
$\tilde\ell$ & sample median of $\{L(C_i)\}_{i=1}^n$ \\
$\pi(C)$ & length penalty, $\min\{1,L(C)/\tilde\ell\}$ \\
\hline
\end{tabular}
\end{center}

\section{Curve band regions and Curve Band Depth}
\label{sec:cbd_def}
\subsection{Band region}
Given $A,B\in\Cset$, define the (global) band region
\begin{equation}
\label{eq:global_band}
\Rband(A,B)=\{(1-\lambda)x+\lambda y:\ x\in A,\ y\in B,\ \lambda\in[0,1]\}.
\end{equation}
Equivalently, $\Rband(A,B)$ is the union of all straight line segments joining a point on $A$ to a point on $B$; it depends only on the traces of $A$ and $B$, hence is parameterization-invariant.

Figure~\ref{fig:cbd_band} illustrates the global band under CBD as well as the corresponding FW-CBD band (which will be introduced in Section~\ref{sec:fwcbd}) on the same pair of curves. In this example, the FW band is visibly a subset of the global convex-combination band.

\begin{figure}[h]
    \centering
    \begin{subfigure}[b]{0.45\textwidth}
        \centering
        \includegraphics[width=\textwidth]{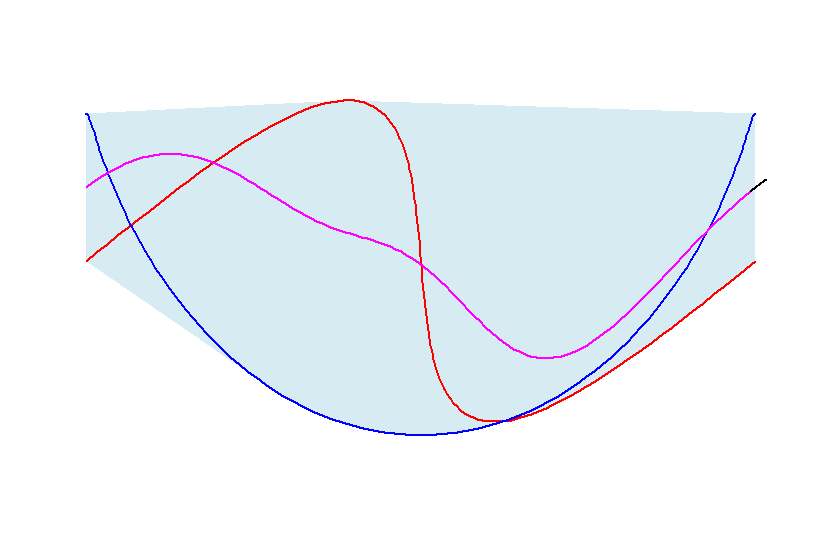}
        \caption{Global band $\Rband(A,B)$.}
    \end{subfigure}
    \hfill
    \begin{subfigure}[b]{0.45\textwidth}
        \centering
        \includegraphics[width=\textwidth]{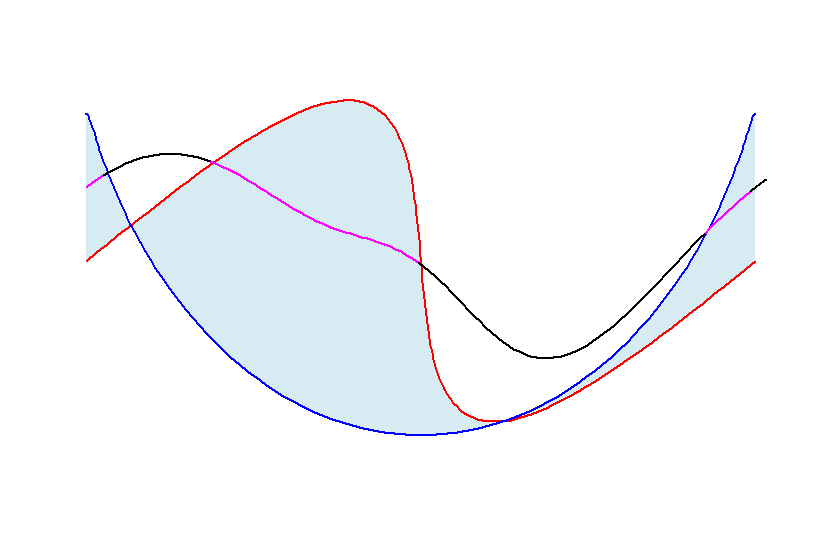}
        \caption{FW band $R_{\mathrm{FW}}(A,B)$.}
    \end{subfigure}
    \caption{The global band of \eqref{eq:global_band} and the FW band of \eqref{eq:fw-band} for the same pair of generating curves $A,B$.}
    \label{fig:cbd_band}
\end{figure}
For any $C\in\Cset$ and band region $\Rband(A,B)$, define in-band proportion as
\begin{equation}
\label{eq:inband}
P(C\mid\Rband(A,B)) := \mu_C(\Rband(A,B)) = \frac{1}{L(C)}\int_C \one_{\Rband(A,B)}(s)\,ds.
\end{equation}
\subsection{Integral and infimal CBD}
Let $\{C_1,\dots,C_n\}\subset\Cset$ be a sample with $n\ge 2$.

\begin{definition}[Integral and infimal curve band depth]
\label{def:int_inf_cbd}
Fix a band construction $\Rband_b(\cdot,\cdot)$. For any target curve $C\in\Cset$, define
\begin{align}
\text{(integral CBD)}\qquad
D_{\mathrm{int}}^{(b)}(C\mid\{C_i\})
&:=
\binom{n}{2}^{-1}
\sum_{1\le i<j\le n}
P\bigl(C\mid \Rband_b(C_i,C_j)\bigr),
\label{eq:int_cbd}
\\
\text{(infimal CBD)}\qquad
D_{\mathrm{inf}}^{(b)}(C\mid\{C_i\})
&:=
\binom{n}{2}^{-1}
\sum_{1\le i<j\le n}
\one\!\left\{\mu_C\bigl(\Rband_b(C_i,C_j)\bigr)=1\right\}.
\label{eq:inf_cbd}
\end{align}
When $\Rband_b$ is the global band $\Rband$ of \eqref{eq:global_band}, we write
$D_{\mathrm{int}}$ and $D_{\mathrm{inf}}$ and refer to them as CBD.
When $\Rband_b$ is the FW band $R_{\mathrm{FW}}$ of Section~\ref{sec:fwcbd}, we write
$D_{\mathrm{FW}}$ and $D_{\mathrm{FW}}^{\inf}$ and refer to them as FW-CBD.
Unless stated otherwise, \emph{CBD} means $D_{\mathrm{int}}$ with the global band.
\end{definition}
\begin{remark}[Why infimal CBD is usually too strict]
The indicator in \eqref{eq:inf_cbd} requires \emph{full} containment of $C$ in the band.
For complex or noisy curves, even a small excursion outside the band makes the indicator zero.
As in functional BD, this can produce many depth values near 0.
Moreover, implementations rely on sampling points along curves, so deciding ``$\mu_C(\Rband)=1$'' is sensitive to discretization.
For these reasons, the present paper focuses on integral CBD as the primary depth.
\end{remark}
\begin{remark}[Population versions]
\label{rmk:population}
If $P$ is a distribution on $\Cset$ and $A,B\stackrel{iid}{\sim}P$, the population integral CBD is
\[
D_{\mathrm{int}}(C\mid P) := \mathbb{E}\big[\mu_C(\Rband(A,B))\big],
\]
and the sample integral CBD in Definition~\ref{def:int_inf_cbd} is the standard $U$-statistic estimator of $D_{\mathrm{int}}(C\mid P)$.
Analogous definitions hold for $D_{\mathrm{inf}}(\cdot\mid P)$ and for the FW versions.
The asymptotic theory of these $U$-statistic estimators (consistency, asymptotic normality, depth-rank stability) is a natural extension of our results and is left as future work.
\end{remark}

\subsection{Fast-walk band region and fast-walk curve band depth}\label{sec:fwcbd}
We now introduce a band construction that is generically a strict subset of the global band of \eqref{eq:global_band}. The fast-walk band is the planar region enclosed by a closed walk formed from the two curves and two closure chords (boundary included). Throughout this subsection curves are assumed rectifiable.
\begin{definition}[Normalized representatives]
\label{def:normapara}
Let $C\in\Cset$ with $L(C)>0$ and let $\beta_C^{\arc}$ be its arc-length representative on $[0,L(C)]$ or $[0,L(C)]/\{0\sim L(C)\}$.
The \emph{normalized representative} is the rescaling
\[
\beta_C(t):=\beta_C^{\arc}\bigl(L(C)\,t\bigr),
\]
defined on $[0,1]$ or $\bbS^1=[0,1]/\{0\sim 1\}$; $\beta_C$ has constant speed $|\beta_C'|=L(C)$ almost everywhere.

If $C$ is open, set
\[
\beta_C^\leftarrow(t):=\beta_C(1-t),\qquad t\in[0,1],
\qquad
\mathcal G(C):=\{\beta_C,\beta_C^\leftarrow\}.
\]

If $C$ is closed, identify $\bbS^1\cong[0,1]/(0\sim1)$ and for $\theta\in\bbS^1$ define the circular shift and its reversal:
\[
\beta_C^\theta(t):=\beta_C(t+\theta \!\!\!\pmod 1),\qquad
(\beta_C^\theta)^\leftarrow(t):=\beta_C(\theta-t \!\!\!\pmod 1),
\qquad
\mathcal G(C):=\bigl\{\beta_C^\theta,\;(\beta_C^\theta)^\leftarrow:\theta\in\bbS^1\bigr\}.
\]
\end{definition}

By Section~\ref{sec:curve_space}, $\beta_C^{\arc}$ is unique up to orientation reversal in the open case and up to circular shift and reversal in the closed
case. Therefore $\beta_C$ inherits the same uniqueness, and $\mathcal G(C)$ depends only on $C\in\Cset$.

\begin{definition}[Fast-walk and enclosed region]\label{def:fwb}
Let $A$ and $B$ be two rectifiable planar curves, and fix
$\alpha\in\mathcal G(A)$ and $\beta\in\mathcal G(B)$.
Define the associated \emph{fast-walk} as the parameterized closed loop
$\gamma_{\alpha,\beta}:\bbS^1\to\R^2$ obtained by concatenating, in order, the four pieces
\begin{equation}
\label{eq:fw-walk}
\alpha\text{ on }[0,1]\;\to\;[\alpha(1),\beta(1)]\;\to\;\beta\text{ in reverse on }[0,1]\;\to\;[\beta(0),\alpha(0)],
\end{equation}
each rescaled to a quarter of $\bbS^1$, so that $\gamma_{\alpha,\beta}$ is continuous and closed.
The boundary trace is the image set
\[
W(\alpha,\beta)
:=\gamma_{\alpha,\beta}(\bbS^1)
=
\alpha([0,1])
\cup [\alpha(1),\beta(1)]
\cup \beta([0,1])
\cup [\beta(0),\alpha(0)],
\]
where $[x,y]$ denotes the closed line segment joining $x$ and $y$.

For every $p\in\R^2\setminus W(\alpha,\beta)$, the \emph{mod-$2$ winding number} of $\gamma_{\alpha,\beta}$ around $p$ is defined as
\begin{equation}
\label{eq:n2-degree}
n_2(\gamma_{\alpha,\beta};p)\;:=\;\deg\!\left(t\mapsto\frac{\gamma_{\alpha,\beta}(t)-p}{\|\gamma_{\alpha,\beta}(t)-p\|}\right)\pmod 2\;\in\;\{0,1\},
\end{equation}
where the right-hand side is the integer Brouwer degree of the indicated continuous map $\bbS^1\to\bbS^1$, reduced modulo $2$.
For polygonal or piecewise-$C^1$ loops in general position with respect to a generic ray from $p$, $n_2(\gamma_{\alpha,\beta};p)$ coincides with the parity of the number of ray-loop intersections.
The enclosed region associated with $(\alpha,\beta)$ is
\begin{equation}
\label{eq:fw-enclosed-region}
E(\alpha,\beta)
:=
W(\alpha,\beta)\cup
\bigl\{\,p\in\R^2\setminus W(\alpha,\beta):\;n_2(\gamma_{\alpha,\beta};p)=1\,\bigr\}.
\end{equation}
That is, $E(\alpha,\beta)$ is the boundary trace $W(\alpha,\beta)$ together with the parity-one complement components of $\R^2\setminus W(\alpha,\beta)$.
\end{definition}

\begin{definition}[Closure score and fast-walk band region] \label{def:closure}
For $\alpha\in\mathcal G(A)$ and $\beta\in\mathcal G(B)$, define the closure score
\[
\Delta(\alpha,\beta)
:=
\|\alpha(1)-\beta(1)\|_2+\|\alpha(0)-\beta(0)\|_2.
\]
Let
\[
\mathcal M(A,B)
:=
\arg\min_{\alpha\in\mathcal G(A),\,\beta\in\mathcal G(B)}
\Delta(\alpha,\beta)
\]
be the set of all closure-minimizing admissible pairs.
The \emph{fast-walk band region} of $A$ and $B$ is then defined by
\begin{equation}
\label{eq:fw-band}
R_{\mathrm{FW}}(A,B)
:=
\cl\!\Bigl(\bigcup_{(\alpha,\beta)\in\mathcal M(A,B)} E(\alpha,\beta)\Bigr).
\end{equation}
\end{definition}

For open curves, the above definition reduces to choosing the better of the two possible endpoint matchings. For closed curves, optimization over $\mathcal G(A)$ and $\mathcal G(B)$ automatically optimizes over all possible starting points; no additional starting-point convention is required.
If several minimizing admissible pairs exist, the definition \eqref{eq:fw-band} takes the closure of the union of the corresponding enclosed regions.

The following lemma is the key structural fact behind FW-CBD. Under the mod-$2$ winding rule, the enclosed region of any closed-curve fast-walk decomposes into a trace-only part plus a measure-zero closure chord. This collapses several invariance proofs in Section~\ref{sec:fw_props} into corollaries.
\begin{lemma}[Parity-interior decomposition for closed curves]
    \label{lemma:closed_parity_decomp}
    Let $A,B\in\Cset$ be closed rectifiable curves with $L(A),L(B)>0$.
    For a closed curve $C\in\Cset$ and $p\in\R^2\setminus\Gamma_C$, define
    \[
    n_2(C;p):=n_2(\eta;p),
    \qquad \eta\in\mathcal G(C),
    \]
    where the right-hand side is independent of the choice of $\eta$.
    Define the mod-$2$ interior of $C$ by
    \[
    I_2(C):=\{p\in\R^2\setminus\Gamma_C:\ n_2(C;p)=1\}.
    \]
    Now set
    \[
    \widehat E_2(A,B)
    :=
    \Gamma_A\cup\Gamma_B\cup\bigl(I_2(A)\triangle I_2(B)\bigr),
    \]
    where $\triangle$ denotes symmetric difference.
    Then, for every $\alpha\in\mathcal G(A)$ and $\beta\in\mathcal G(B)$,
    \[
    E(\alpha,\beta)=\widehat E_2(A,B)\cup[\alpha(0),\beta(0)].
    \]
    In particular, $\widehat E_2(A,B)$ depends only on the curve classes
    $A$ and $B$, and not on the starting points or orientations chosen in
    $\alpha$ and $\beta$.
 \end{lemma}

\begin{proof}
We first check that $n_2(C;p)$ is well-defined. Replacing $\eta$ by a circular shift only reparameterizes the same loop, so the mod-$2$ degree is unchanged. Replacing $\eta$ by its reversal changes the integer degree by a sign, but the sign disappears modulo $2$. Hence $n_2(C;p)$ is independent of the choice of $\eta\in\mathcal G(C)$.

Fix $\alpha\in\mathcal G(A)$ and $\beta\in\mathcal G(B)$. Because $A$ and $B$ are closed, $\alpha(0)=\alpha(1)$ and $\beta(0)=\beta(1)$. The two closure chords coincide as the single segment
$\Lambda_{\alpha,\beta}:=[\alpha(0),\beta(0)],$ and $W(\alpha,\beta)=\Gamma_A\cup\Gamma_B\cup \Lambda_{\alpha,\beta}.$ Let $\sigma$ be a parameterization of $\Lambda_{\alpha,\beta}$ from $\alpha(0)$ to $\beta(0)$. Then the fast-walk loop can be written as
    \[
    \gamma_{\alpha,\beta}
    =
    \alpha * \sigma * \beta^\leftarrow * \sigma^{-1}.
    \]
Fix $p\in\R^2\setminus W(\alpha,\beta)$. By additivity of mod-$2$ degree under concatenation,
\[
    n_2(\gamma_{\alpha,\beta};p)
    \equiv
    n_2(\alpha;p)
    +
    n_2(\sigma * \beta^\leftarrow * \sigma^{-1};p)
    \pmod 2.
\]
Since $p\notin \Lambda_{\alpha,\beta}$, the path $\sigma*\sigma^{-1}$ is null-homotopic in $\R^2\setminus\{p\}$, and therefore $\sigma * \beta^\leftarrow * \sigma^{-1}$ is freely homotopic to $\beta^\leftarrow$ in $\R^2\setminus\{p\}$. Free homotopy preserves the mod-$2$ degree, and reversal does not change it.
Hence
    \[
    n_2(\sigma * \beta^\leftarrow * \sigma^{-1};p)
    \equiv n_2(\beta;p)\pmod 2.
    \]
    Consequently,
    \[
    n_2(\gamma_{\alpha,\beta};p)
    \equiv
    n_2(A;p)+n_2(B;p)
    \pmod 2.
    \]
By the definition of $I_2(A)$ and $I_2(B)$, this is equal to $1$ precisely when $p\in I_2(A)\triangle I_2(B).$ Thus
    \[
    \{p\in\R^2\setminus W(\alpha,\beta):
    n_2(\gamma_{\alpha,\beta};p)=1\}
    =
    \bigl(I_2(A)\triangle I_2(B)\bigr)\setminus W(\alpha,\beta).
    \]
    Recall Equation~\eqref{eq:fw-enclosed-region}, there is
    \[
    E(\alpha,\beta)
    =
    W(\alpha,\beta)\cup\bigl(I_2(A)\triangle I_2(B)\bigr)
    =
    \widehat E_2(A,B)\cup[\alpha(0),\beta(0)].
    \]
Since $I_2(A)$ and $I_2(B)$ are intrinsic to $A$ and $B$, the set $\widehat E_2(A,B)$ is independent of the choices of $\alpha$ and $\beta$.
\end{proof}
Figure~\ref{fig:cbdfw_starts} illustrates Lemma~\ref{lemma:closed_parity_decomp} on two nested off-centre closed curves: three different starting-point pairs produce three visibly different closure chords and three different closure scores $\Delta$, but, modulo the chord segment, the enclosed region coincides with $\widehat E_2(A,B)$ in all three. Only the first starting-point pair $(\alpha(0),\beta(0))$  attains the minimum $\Delta$ value and is admissible in $\mathcal M(A,B)$, hence is the canonical walk representative; the other two are not.

\begin{figure}[h]
    \centering
    \includegraphics[width=\textwidth]{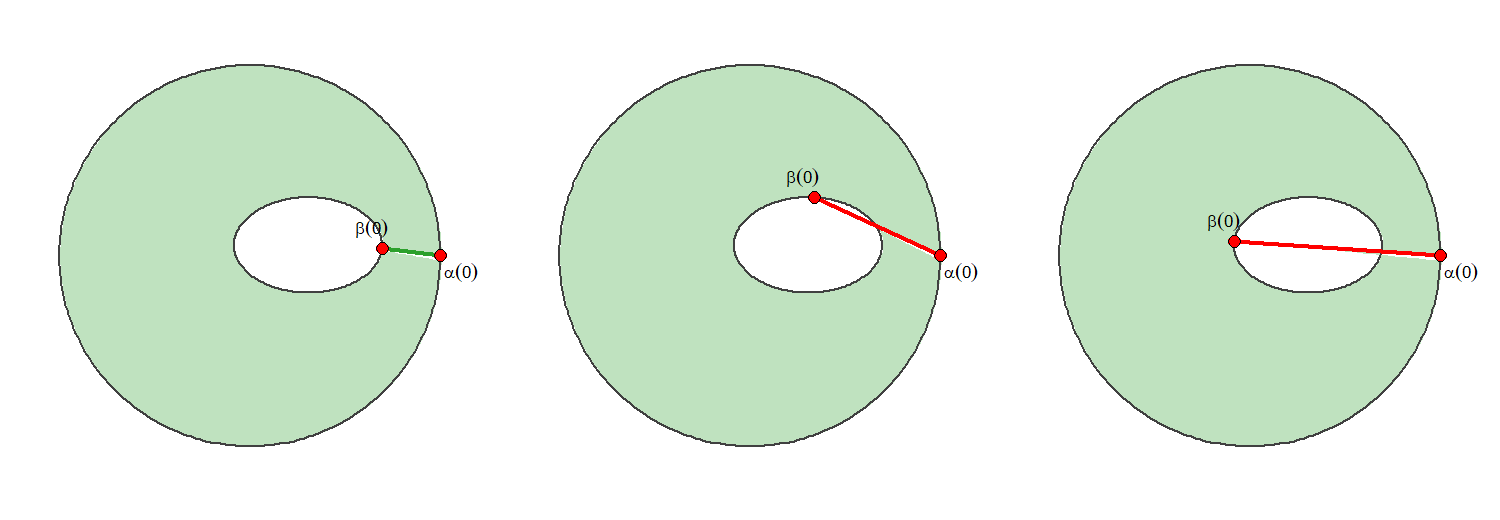}
    \caption{Closed-curve illustration of the starting-point selection in Definition~\ref{def:closure}, for nested closed curves (outer circle $A$, inner ellipse $B$).
    Each panel uses a different starting-point pair $(\alpha(0),\beta(0))$, producing a visibly different closure chord across the asymmetric annulus and a different closure score $\Delta(\alpha,\beta)$.
    Up to the measure-zero chord segment, the shaded region is the same across
    all three.}
    \label{fig:cbdfw_starts}
\end{figure}

By Definition~\ref{def:int_inf_cbd}, the FW band $R_{\mathrm{FW}}$ induces the
corresponding integral and infimal depths. In this case we write
$D_{\mathrm{FW}}$ and $D_{\mathrm{FW}}^{\inf}$ for the resulting sample depths,
and
\[
P_{\mathrm{FW}}(C\mid A,B):=\mu_C\bigl(R_{\mathrm{FW}}(A,B)\bigr)
\]
for the associated in-band proportion.

\subsection{Length-penalized CBD}
\label{sec:cbd_pen}
The integral CBD favors shorter curves: a target curve that is much shorter than typical sample curves is more likely to lie entirely inside any given band.
This is benign when curve lengths are comparable, but becomes a confound on data sets such as hand-written letters in which one class has systematically shorter trajectories than another.
We therefore introduce a multiplicative penalty based on length relative to a robust reference (i.e., the sample median).

Let $\tilde\ell$ denote the sample median of $\{L(C_i)\}_{i=1}^n$ and define the penalty factor
\[
\pi(C) := \min\{1,\ L(C)/\tilde\ell\}.
\]
The penalized CBD and penalized FW-CBD are then
\[
D^{\mathrm{pen}}_{\mathrm{int}}(C) := \pi(C)\,D_{\mathrm{int}}(C),
\qquad
D^{\mathrm{pen}}_{\mathrm{FW}}(C) := \pi(C)\,D_{\mathrm{FW}}(C).
\]
Other penalty shapes can be considered. The one above preserves similarity invariance at the dataset level (Proposition~\ref{prop:similarity}). If the same similarity transformation $g$ is applied to every curve in the sample $\{C_i\}_{i=1}^n$ and to the target curve $C$, both $L(C)$ and the median $\tilde\ell$ scale by the common factor $r$, so the ratio $L(C)/\tilde\ell$ and hence $\pi(C)$ are unchanged.
The penalized depth is therefore similarity-invariant when the sample is co-transformed with the target curve.

\begin{remark}[CBD family closure]
\label{rmk:cbd_family}
The construction of Definition~\ref{def:int_inf_cbd} extends to any band map satisfying the following two structural conditions:
(i) for every $A,B\in\Cset$, $\Rband(A,B)$ is a Borel subset of $\R^2$;
(ii) $\Rband(A,B)$ depends only on the equivalence classes $A,B\in\Cset$, i.e., is invariant under monotone reparameterization of $A$ and $B$.
Together these ensure that the in-band proportion $\mu_C(\Rband(A,B))$ is well-defined and that, under the additional hypothesis of Theorem~\ref{thm:measurability}, the resulting integral and infimal depths are Borel-measurable in $C$.
The global band $\Rband$ of \eqref{eq:global_band} satisfies (i)--(ii) trivially; the FW band of Section~\ref{sec:cbd_def} satisfies them by Propositions~\ref{prop:fw_welldef} and \ref{prop:fw_param}.
The length-penalized variants are not new bands but multiplicative adjustments of CBD by a scalar functional that is itself invariant under monotone reparameterization, so they inherit measurability and reparameterization-invariance from the underlying band.
\end{remark}

\section{Properties}
\label{sec:properties}
We state the properties of CBD and FW-CBD jointly whenever the same argument applies to both bands.
Section~\ref{sec:fw_props} lists the structural properties that are specific to the FW band.
\begin{remark}[FW-band depth prerequisites]
Propositions~\ref{prop:fw_welldef} and \ref{prop:fw_param} (proved in Section~\ref{sec:fw_props}) establish that $\mathcal M(A,B)\ne\emptyset$ and that $R_{\mathrm{FW}}(A,B)$ is a well-defined, compact, Borel function of the equivalence classes $(A,B)$.
These facts justify the use of $R_{\mathrm{FW}}(A,B)$, $\mu_C(R_{\mathrm{FW}}(A,B))$, $D_{\mathrm{FW}}$, and $D_{\mathrm{FW}}^{\mathrm{inf}}$ in Proposition~\ref{prop:boundedness} and Theorem~\ref{thm:measurability}.
\end{remark}
\subsection{Boundedness}
\begin{proposition}[Boundedness]
\label{prop:boundedness}
For any $C\in\Cset$ and any sample $\{C_i\}_{i=1}^n$, the four sample depths
\(D_{\mathrm{int}}(C\mid\{C_i\}),\;
  D_{\mathrm{inf}}(C\mid\{C_i\}),\;
  D_{\mathrm{FW}}(C\mid\{C_i\}),\;
  D_{\mathrm{FW}}^{\inf}(C\mid\{C_i\})\)
all lie in $[0,1]$.
\end{proposition}
\begin{proof}
For either band $\widehat\Rband\in\{\Rband,R_{\mathrm{FW}}\}$ and any pair $(i,j)$,
$0\le \mu_C(\widehat\Rband(C_i,C_j))\le 1$ because $\mu_C$ is a probability measure.
Each integral depth is a $\binom{n}{2}^{-1}$-weighted average of these quantities and hence lies in $[0,1]$.
Each infimal depth is a $\binom{n}{2}^{-1}$-weighted average of indicators in $\{0,1\}$ and hence also lies in $[0,1]$.
\end{proof}

\subsection{Vanishing at infinity}

For a fixed band construction $\Rband_b\in\{\Rband,R_{\mathrm{FW}}\}$, define
\[
S_b := \bigcup_{1\le i<j\le n}\Rband_b(C_i,C_j).
\]
Only the portion of a target curve lying in $S_b$ can contribute to the
corresponding integral depth.

\begin{proposition}[Vanishing by escape in arc-length proportion]
\label{prop:vanishing_proportion}
For any $C\in\Cset$ and any band construction
$\Rband_b\in\{\Rband,R_{\mathrm{FW}}\}$, $
D_{\mathrm{int}}(C\mid\{C_i\})
\le
\mu_C(S_b).
$ Consequently, if $\{C^{(k)}\}\subset\Cset$ satisfies $\mu_{C^{(k)}}(S_b)\to 0$, 
then $
D_{\mathrm{int}}(C^{(k)}\mid\{C_i\})\to 0.
$
Moreover, since $\mu_{C^{(k)}}(S_b)<1$ for all sufficiently large $k$, the
corresponding infimal depth satisfies $
D_{\mathrm{inf}}(C^{(k)}\mid\{C_i\})=0
$
for all sufficiently large $k$.
\end{proposition}

\begin{proof}
For every pair $(i,j)$,
$
\Rband_b(C_i,C_j)\subseteq S_b.
$ Hence, $
\mu_C\bigl(\Rband_b(C_i,C_j)\bigr)\le \mu_C(S_b).
$ Averaging over all unordered pairs gives
$
D_{\mathrm{int}}^{(b)}(C\mid\{C_i\})
=
\binom{n}{2}^{-1}
\sum_{1\le i<j\le n}
\mu_C\bigl(\Rband_b(C_i,C_j)\bigr)
\le
\mu_C(S_b).
$ This proves the first claim, and the convergence follows immediately. For the infimal depth, if $\mu_C(S_b)<1$, then no pairwise band
$\Rband_b(C_i,C_j)\subseteq S_b$ can satisfy
$
\mu_C\bigl(\Rband_b(C_i,C_j)\bigr)=1.
$ Thus every infimal indicator is zero, and therefore
$D_{\mathrm{inf}}^{(b)}(C\mid\{C_i\})=0$.
Applying this to $C=C^{(k)}$ gives the eventual vanishing of the infimal depth.
\end{proof}

\begin{remark}[Distance sufficient conditions] \label{rmk:vanishing_distance}
The condition $\mu_{C^{(k)}}(S_b)\to0$ is an arc-length proportion condition. A direct sufficient condition is
$ \operatorname{dist}(C^{(k)},S_b) := \inf_{x\in C^{(k)},\,z\in S_b}\|x-z\| \to\infty.$ Indeed, since $S_b$ is a finite union of compact bands, it is compact; hence $C^{(k)}\cap S_b=\varnothing$ for all sufficiently large $k$, and therefore $\mu_{C^{(k)}}(S_b)=0$ eventually.
    
A weaker-looking condition based only on a farthest point is not sufficient by itself, because the depth is based on arc-length proportion. It becomes sufficient under a uniform length bound. If
    $
    \sup_k L(C^{(k)})<\infty
    \quad\text{and}\quad
    \sup_{x\in C^{(k)}}\operatorname{dist}(x,S_b)\to\infty,
    $
    then
    $
    \operatorname{dist}(C^{(k)},S_b)\to\infty.
    $
    To see this, let $x_k\in C^{(k)}$ attain the supremum distance. For any
    $y\in C^{(k)}$,
    $
    \operatorname{dist}(y,S_b)
    \ge
    \operatorname{dist}(x_k,S_b)-\|x_k-y\|
    \ge
    \operatorname{dist}(x_k,S_b)-L(C^{(k)}).
    $
    The right-hand side tends to infinity under the stated assumptions.
\end{remark}
    
\subsection{Similarity invariance}
Let $g:\R^2\to\R^2$ be a similarity transformation $g(x)=rQx+b$ with scale $r>0$, orthogonal $Q$, and translation $b\in\R^2$.

\begin{lemma}[Band equivariance]
\label{lemma:band_equivariance}
Let $g(x)=rQx+b$ be a similarity of $\R^2$ with scale $r>0$, orthogonal $Q\in O(2)$, and translation $b\in\R^2$.
For any $A,B\in\Cset$,
\[
g(\Rband(A,B))=\Rband(g(A),g(B))
\qquad\text{and}\qquad
g(R_{\mathrm{FW}}(A,B))=R_{\mathrm{FW}}(g(A),g(B)).
\]
\end{lemma}
\begin{proof}
For the global band, if $z=(1-\lambda)x+\lambda y\in\Rband(A,B)$ with $x\in A,\;y\in B,\;\lambda\in[0,1]$, then
$g(z)=(1-\lambda)g(x)+\lambda g(y)\in\Rband(g(A),g(B))$ because $g$ is affine.
The converse follows by applying $g^{-1}$. The proof of FW identity is given in Appendix~\ref{app:proofs}.
\end{proof}

\begin{proposition}[Similarity invariance]
\label{prop:similarity}
For any $C\in\Cset$, any sample $\{C_i\}_{i=1}^n$, and any similarity $g$,
\[
D_{\bullet}(C\mid\{C_i\})=D_{\bullet}(g(C)\mid\{g(C_i)\})
\quad\text{for}\quad
D_{\bullet}\in\bigl\{D_{\mathrm{int}},D_{\mathrm{inf}},D_{\mathrm{FW}},D_{\mathrm{FW}}^{\inf}\bigr\}.
\]
\end{proposition}
\begin{proof}
Fix any band $\widehat\Rband\in\{\Rband,R_{\mathrm{FW}}\}$.
By Lemma~\ref{lemma:band_equivariance}, $g(\widehat\Rband(C_i,C_j))=\widehat\Rband(g(C_i),g(C_j))$.
A similarity scales arc-length by $r$: $L(g(C))=r\,L(C)$, and an arc-length representative of $g(C)$ is $s\mapsto g(\beta_C^{\arc}(s/r))$ on $[0,r\,L(C)]$ (speed $|(g(\beta_C^{\arc}(s/r)))'|=r\cdot|\beta_C^{\arc}{}'(s/r)|\cdot(1/r)=1$ a.e.).
Hence for any Borel set $B\subseteq\R^2$,
\[
\beta^{\arc}_{g(C)}(s)=g\!\left(\beta_C^{\arc}(s/r)\right),\qquad s\in[0,rL(C)].
\]
Hence, for any Borel set $B\subset\R^2$,
\[
\mu_{g(C)}(B)
=\frac{1}{rL(C)}\int_0^{rL(C)}
\one_B\!\left(g\!\left(\beta_C^{\arc}(s/r)\right)\right)\,ds
=\frac{1}{L(C)}\int_0^{L(C)}
\one_{g^{-1}(B)}\!\bigl(\beta_C^{\arc}(u)\bigr)\,du
=\mu_C(g^{-1}B),
\]
where the second equality follows by the change of variables $u=s/r$.
Applying this with $B=\widehat\Rband(g(C_i),g(C_j))$ and using $g^{-1}(\widehat\Rband(g(C_i),g(C_j)))=\widehat\Rband(C_i,C_j)$, we get
\[
\mu_{g(C)}\bigl(\widehat\Rband(g(C_i),g(C_j))\bigr)=\mu_C\bigl(\widehat\Rband(C_i,C_j)\bigr),
\]
which gives the equalities for the integral and infimal depths after averaging over unordered pairs.
\end{proof}

\subsection{Measurability}
\label{sec:measurability}
Throughout this subsection, $\Cset$ carries the Borel $\sigma$-algebra induced by the embedding $\Phi$ of Section~\ref{sec:curve_space}; set operations on $C\in\Cset$ use the trace identification of Section~\ref{sec:curve_space}, and the Hausdorff metric $d_H$ on the hyperspace $(\mathcal{K},d_H)$ of nonempty compact subsets of $\R^2$ is used only when needed to metrize the trace map $C\mapsto\Gamma_C$.

\begin{theorem}[Measurability of the CBD map]
\label{thm:measurability}
Fix a finite sample $\{C_i\}_{i=1}^n\subset\Cset$ with $n\ge 2$.
Assume the arc-length probability measure $\mu_C$ of \eqref{eq: arc-length} is defined for every $C\in\Cset$, and that for every Borel set $B\subseteq\R^2$ the map $C\mapsto\mu_C(B)$ is Borel-measurable with respect to the Borel $\sigma$-algebra on $\Cset$ induced by $\Phi$.
Then all four depth maps
\[
C\mapsto D_{\bullet}(C\mid\{C_i\}),\qquad
\bullet\in\{\mathrm{int},\,\mathrm{inf},\,\mathrm{FW},\,\mathrm{FW}^{\mathrm{inf}}\},
\]
are Borel-measurable with respect to the same $\sigma$-algebra.
\end{theorem}
\begin{proof}
For fixed $i<j$, each of the two bands for the pair $(C_i,C_j)$ is a compact, hence Borel, subset of $\R^2$: the global band $\Rband(C_i,C_j)$ is the continuous image of $C_i\times C_j\times[0,1]$ under the convex-combination map, and the FW band $R_{\mathrm{FW}}(C_i,C_j)$ is compact by Proposition~\ref{prop:fw_welldef}.
By the standing assumption, $C\mapsto\mu_C(B)$ is Borel-measurable for every Borel $B$; applying this to $B=\Rband(C_i,C_j)$ and $B=R_{\mathrm{FW}}(C_i,C_j)$ gives Borel-measurability of $C\mapsto\mu_C(B)$ for each fixed pair.
The integral depths $D_{\mathrm{int}}$ and $D_{\mathrm{FW}}$ are $\binom{n}{2}^{-1}$-weighted finite sums of these measurable maps, hence Borel-measurable.
For the infimal depths, observe that for any compact (hence closed) band $B\subset\R^2$ and any $C\in\Cset$:
$C\subseteq B \Leftrightarrow \mu_C(B)=1$.
(The direction $\Rightarrow$: $\beta_C^{\arc}$ takes values in $C\subseteq B$, so the integrand $\one_B(\beta_C^{\arc}(s))=1$ a.e. The direction $\Leftarrow$: if $C\not\subseteq B$, then by continuity of $\beta_C^{\arc}$ there is an open arc of positive length mapped outside $B$, giving $\mu_C(B)<1$.)
Hence $D_{\mathrm{inf}}$ and $D_{\mathrm{FW}}^{\mathrm{inf}}$ can each be expressed using $1\{\mu_C(B)=1\}$, the indicator of the preimage of $\{1\}$ under a measurable map; such indicators are measurable, and summing and scaling preserves measurability.
\end{proof}

\begin{remark}[Scope of the measurability assumption]
\label{rmk:meas_scope}
The hypothesis that $C\mapsto\mu_C(B)$ be Borel-measurable with respect to the $\Phi$-induced Borel $\sigma$-algebra on $\Cset$ is non-trivial and is treated as a standing assumption here. Establishing this measurability in full generality remains a natural direction for future work.
\end{remark}

\subsection{FW-specific structural properties}
\label{sec:fw_props}
The propositions in this section concern structural properties of the FW band that have no direct counterpart for the global band.
Detailed proofs are deferred to Appendix~\ref{app:proofs}. As noted in Remark~\ref{rmk:degenerate_nonatomic}, we assume that $A,B\in\Cset$ satisfy $L(A),L(B)>0$.

\begin{proposition}[Well-definedness]
\label{prop:fw_welldef}
For any two rectifiable curves $A$ and $B$, the minimizing set $\mathcal M(A,B)$ is nonempty and the fast-walk band region $R_{\mathrm{FW}}(A,B)$ is compact and Borel-measurable.
Consequently, $P_{\mathrm{FW}}(C\mid A,B)$ is well-defined for every rectifiable curve $C$.
\end{proposition}

\begin{proposition}[Reparameterization invariance of the FW band]
\label{prop:fw_param}
For each $A,B\in\Cset$, the admissible family $\mathcal G(A)$ depends only on the equivalence class $A\in\Cset$ and not on the particular constant-speed representative chosen in Definition~\ref{def:normapara}; the same holds for $B$.
Consequently, $R_{\mathrm{FW}}(A,B)$ is a well-defined function of $(A,B)\in\Cset\times\Cset$ — equivalently, $R_{\mathrm{FW}}$ is invariant under monotone reparameterization of $A$ and $B$.
\end{proposition}

\begin{proposition}[Symmetry in the two generating curves]
\label{prop:fw_sym}
For any two rectifiable curves $A$ and $B$, $R_{\mathrm{FW}}(A,B)=R_{\mathrm{FW}}(B,A)$.
\end{proposition}

\begin{proposition}[FW band is contained in the global band]
    \label{prop:fw_subset_global}
    For any two rectifiable curves $A,B\in\Cset$,
    \[
    R_{\mathrm{FW}}(A,B)\subseteq \Rband(A,B).
    \]
\end{proposition}

\section{Implementation}\label{sec:implementation}

In practice, each curve is approximated by $m$ points sampled at roughly equal
arc-length intervals. Denote by
\[
\tilde C_i=\{x_{i,1},\dots,x_{i,m}\}
\]
the resulting $m$-point discretization of $C_i$.
The same sampling scheme is used for target curves.

\subsection{Global band approximation}
\label{sec:impl_global}

\paragraph{Tiling construction.}
The most direct discrete realization of the global band
$\Rband(\tilde C_i,\tilde C_j)$ tiles the region between the two discretized curves
with the quadrilaterals
\[
\bigl[x_{i,v},x_{i,v+1},x_{j,u+1},x_{j,u}\bigr],
\qquad v,u=1,\dots,m-1,
\]
splits ill-defined quadrilaterals into triangles, and forms the resulting
(possibly non-convex) band polygon by polygonal union.
For a target curve $\tilde C_k$, the in-band fraction is then computed by
point-in-polygon tests for the $m$ sampled points $\{x_{k,q}\}_{q=1}^m$.
This realization follows Definition~\ref{def:int_inf_cbd} directly and is a
natural discrete approximation to $\mu_C(\Rband(C_i,C_j))$ as the sampling
resolution $m$ increases. However, the
polygon-union step makes it too costly, so we do
not use it as the primary estimator.

\paragraph{Chord-raster approximation.}
The practical estimator used in our experiments for the global band is a
pixel-stamp approximation of $\Rband(\tilde C_i,\tilde C_j)$.
We discretize the joint bounding box of the reference set on a uniform
$G\times G$ grid.
For each pair $(i,j)$, we draw $N$ chord endpoints
$\{(u_s,v_s)\}_{s=1}^N$ uniformly from $\{1,\dots,m\}^2$, rasterize each chord
$[x_{i,u_s},x_{j,v_s}]$ onto the grid by integer line interpolation, and stamp a
disk of integer radius $\rho$ pixels around every chord pixel.
Let $S_{ij}\subset\{1,\dots,G\}^2$ denote the set of stamped cells.
We then approximate
\[
\Rband(\tilde C_i,\tilde C_j)
\;\approx\;
\bigcup_{(p_x,p_y)\in S_{ij}} \mathrm{cell}(p_x,p_y).
\]
For a target curve $\tilde C_k$, the corresponding in-band fraction is
\[
\widehat P_{\mathrm{cr}}(\tilde C_k\mid\Rband(\tilde C_i,\tilde C_j))
=
\frac{1}{m}\sum_{q=1}^m
\one\!\bigl(\mathrm{cell}(x_{k,q})\in S_{ij}\bigr).
\]
The continuous definition satisfies $A\subseteq\Rband(A,B)$ and $B\subseteq\Rband(A,B)$ exactly, but the pixel-stamp approximation can miss isolated sampled points of the generating curves due to grid discretization.
We therefore override the estimator for the two generating indices, setting
$\widehat P_{\mathrm{cr}}(\tilde C_k\mid\Rband(\tilde C_i,\tilde C_j))=1$ whenever $k\in\{i,j\}$, so that self-inclusion holds exactly for the discrete estimator.

\subsection{FW band approximation}
\label{sec:impl_fw}
The FW estimator realizes the theoretical band $R_{\mathrm{FW}}(A,B)$ of
\eqref{eq:fw-band} through an aligned walk-polygon approximation.
Given arc-length samples
\[
\tilde C_i=\{x_{i,1},\dots,x_{i,m}\},
\qquad
\tilde C_j=\{x_{j,1},\dots,x_{j,m}\},
\]
we enumerate the discrete family induced by $\tilde C_j$, $\mathcal G(\tilde C_j)$. For each candidate
\[
\tilde C_j^{(c)}=\{y_1^{(c)},\dots,y_m^{(c)}\},
\]
we compute the discrete closure score
\[
\widehat\Delta\bigl(\tilde C_i,\tilde C_j^{(c)}\bigr)
:=
\|x_{i,m}-y_m^{(c)}\|_2+\|y_1^{(c)}-x_{i,1}\|_2,
\]
which is the polygonal analogue of the continuous closure score.
Let $\mathcal M^\ast$ denote the set of all discrete minimizers. For each $c\in\mathcal M^\ast$, we form the closed-walk polygon
\begin{equation}
\label{eq:fw-impl-poly}
P_{ij}^{(c)}
:=
\bigl(x_{i,1},\dots,x_{i,m},\;y_m^{(c)},\dots,y_1^{(c)},\;x_{i,1}\bigr),
\end{equation}
whose closing edges are the two closure chords.
We then define $\widehat R_{\mathrm{FW}}(\tilde C_i,\tilde C_j)$ as the union of the polygonal regions enclosed by all $P_{ij}^{(c)}$, and determine target-point membership via \texttt{sp::point.in.polygon} in \texttt{R} \citep{sp:bivand2013applied}.

\subsection{In-band proportions}
\label{sec:impl_inband}

Given any of the discrete band approximations above, the in-band fraction for a
target curve $\tilde C_k$ is computed by point-in-set tests on the $m$ sampled
points,
\[
\widehat P(\tilde C_k\mid\widehat\Rband(\tilde C_i,\tilde C_j))
=
\frac{1}{m}\sum_{q=1}^m
\one\!\bigl(x_{k,q}\in\widehat\Rband(\tilde C_i,\tilde C_j)\bigr).
\]
The corresponding sample depth is then obtained by averaging $\widehat P$ over all
unordered pairs $(i,j)$.

\subsection{Approximation accuracy and discretization effects}
\label{sec:impl_accuracy}

The discrete estimators above approximate the arc-length probability
$\mu_C(\Rband(A,B))$ or $\mu_C(R_{\mathrm{FW}}(A,B))$ by evaluating membership on an equi-arc-length discretization of the target curve.
Under deterministic equal-arc-length sampling, this is a Riemann-type approximation to the corresponding line integral.
Accordingly, increasing $m$ stabilizes the estimated in-band fraction for both the global-band and FW-band constructions.

The chord-raster approximation introduces an additional source of error that does not arise in the tiling or FW implementations.
Two effects are relevant. First, the stamp set $S_{ij}$ is built by union, so once a pixel cell is marked it remains marked, and additional sampled chords can only enlarge the apparent band.
Second, the stamp radius $\rho$ is implemented on the discrete grid, which induces a resolution floor of order $1/G$ in physical units.
Together, these effects produce a small upward bias that does not disappear merely by increasing the number of sampled chords.

To quantify these approximation effects, we consider a toy configuration for which the finite-sample CBD can be computed exactly.
Specifically, we generate $50$ unit-length rays emanating from a common origin, with directions given by $25$ random angles in $(0,\pi/4)$ together with their antipodal counterparts.
Each curve is therefore a line segment from the origin to a point on the unit circle.

This configuration is convenient because, for any pair of rays, the global
band $\Rband(C_i,C_j)$ is exactly the filled triangle spanned by the common origin and the two endpoints.
Consequently, for any target ray $C_k$, the in-band proportion
$\mu_{C_k}(\Rband(C_i,C_j))$ admits a closed-form expression determined by planar triangle geometry, and the resulting finite-sample CBD values can therefore be computed exactly.
In this unit-ray setting, the FW band coincides with the global band, because the fast-walk boundary formed by two rays together with the closure chords encloses exactly the same triangular region as the global band. Thus both the FW implementation and the chord-raster implementation can be benchmarked against the same exact finite-sample CBD values. Figure~\ref{fig:unit-ray-config} displays one realization of this toy configuration.

\begin{figure}[htbp]
\centering
\includegraphics[width=.62\textwidth]{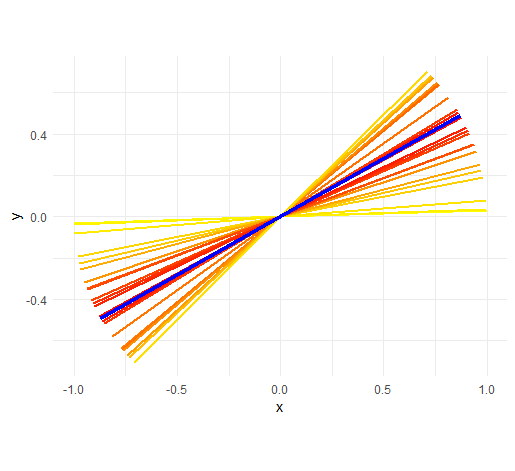}
\caption{A realization of the unit-ray toy configuration used for the approximation study.
Each curve is a unit-length ray emanating from the common origin. Colors indicate the
depth ranking, from yellow (shallower) to red (deeper), and the deepest curve is shown
in blue.}
\label{fig:unit-ray-config}
\end{figure}

For each random seed, we compute the absolute deviation
\(
\bigl|\widehat D(C_k)-D_{\mathrm{CBD}}^{\mathrm{theo}}(C_k)\bigr|
\)
for all $50$ rays, and then pool these deviations over $10$ random seeds.
Table~\ref{tab:cbd_bias} reports the pooled mean absolute errors for the FW estimator and for the chord-raster estimator under several discretization settings.

\begin{table}[h!]
\centering
\caption{Pooled mean absolute error of the FW and chord-raster estimators against the theoretical CBD on the unit-ray toy configuration, aggregated over $10$ random seeds and $50$ curves. For the chord-raster estimator, the pixel grid is fixed at $G=600$ and the stamp radius at $\rho=1$ pixel.}
\label{tab:cbd_bias}
\begin{tabular}{lccc}
\hline
\textbf{Estimator} & $m_{\mathrm{eval}}$ & $N$ chords & Mean abs.\ error \\
\hline
FW                 & 120  & --   & $5.92\times10^{-3}$ \\
FW                 & 500  & --   & $1.43\times10^{-3}$ \\
FW                 & 1000 & --   & $7.13\times10^{-4}$ \\
chord-raster       & 120  & 1000 & $3.08\times10^{-2}$ \\
chord-raster       & 500  & 1000 & $2.85\times10^{-2}$ \\
chord-raster       & 1000 & 1000 & $2.83\times10^{-2}$ \\
chord-raster       & 500  & 3000 & $3.07\times10^{-2}$ \\
chord-raster       & 500  & 5000 & $3.13\times10^{-2}$ \\
\hline
\end{tabular}
\end{table}

Table~\ref{tab:cbd_bias} shows two things.
First, the FW approximation error decreases rapidly as the arc-length sampling
resolution $m_{\mathrm{eval}}$ increases, because in this geometry the FW polygon is a direct discretization of the exact triangular band.
Second, the chord-raster error stabilizes around $3\times 10^{-2}$ and does not
improve when more chords are sampled. The fixed pixel grid imposes a resolution
floor, and the stamp set $S_{ij}$ is built by union. Once a pixel cell is marked, it remains marked, so additional chords can only preserve or enlarge the apparent band. Thus, after the relevant cells have been saturated, additional sampled chords add little new geometric information and may even slightly increase the error through extra over-coverage. In this example, the dominant source of chord-raster error is therefore the grid resolution rather than Monte Carlo variability.
\subsection{Computational complexity}
\label{sec:impl_complexity}
Let $n$ denote the sample size, $m$ the number of equi-arc-length sample points per
curve, and $N$ the number of sampled chords in the chord-raster approximation.
Since all depths average over unordered curve pairs, the number of generating pairs is
$\binom{n}{2}=O(n^2)$.

At the pairwise band-construction level, the dominant costs are: tiling, $O(m^2)$ before polygonal union, where the union step is the main computational bottleneck in practice; chord-raster, $O(N)$ when the grid resolution and stamp radius are treated as fixed; FW, $O(m)$ for constructing the walk polygon when both curves are open (only two orientations need be checked), and up to $O(m^2)$ when at least one curve is closed if all circular shifts are enumerated. In all three cases, one must subsequently evaluate band membership for the $m$ sampled points on each target curve in order to compute the in-band proportion.

In practice, exact tiling is the slowest because of the polygonal union overhead, whereas FW is typically the fastest for open curves.

\section{Numerical illustrations and experiments}
\label{sec:experiments}
We first illustrate the proposed depths on hand-crafted toy curves for which the band geometry is fully transparent. We then evaluate them on three tasks: outlier detection on hand-written letters from \texttt{mrfDepth}~\citep{mrfD}, curve classification on hand-written letters and MNIST digit trajectories~\citep{deng2012mnist, de2021depth}, and a one-class geometric screening task on the \texttt{MOBISIG} online-signature dataset~\citep{antal2018mobisig}. Finally, we give a small exploratory clustering analysis on UCI Character Trajectories~\citep{williams2006character}.
Throughout, CBD is computed with the chord-raster estimator at $G=600$ and stamp radius $\rho=1$ pixel, and the halfspace curve depth (HCD) of \citet{de2021depth} serves as the comparison method.

\subsection{Toy examples}
We compare CBD, FW-CBD, and HCD on four toy configurations: heterogeneous parallel segments, an equal-length star, a restricted-angle star, and concentric circles.

The four columns of Figure~\ref{fig:cbdtoy} illustrate where the three methods agree and where they diverge. For parallel segments, all three favor curves near the middle of the sample. Without length adjustment, however, CBD and FW-CBD also prefer shorter segments, because shorter curves are more easily covered by pairwise bands. In the equal-length star, all three assign essentially the same depth to every curve; this is forced by the symmetry of the configuration together with similarity invariance. In the restricted-angle star, all three again favor rays near the angular center of the observed wedges, with qualitatively similar rankings.

The clearest disagreement is on concentric circles. CBD ranks the innermost circles as deepest, since every pairwise band contains the inner ones. FW-CBD picks a circle whose radius is roughly in the middle of the observed range, because the FW band behaves like a radial interval between two circles. HCD also attains its maximum at an interior radius, between the two but visually closer to the FW-CBD choice.

We now move to real data.

\begin{figure}[h]
    \begin{subfigure}[b]{0.24\textwidth}
        \centering
        \includegraphics[width=\textwidth]{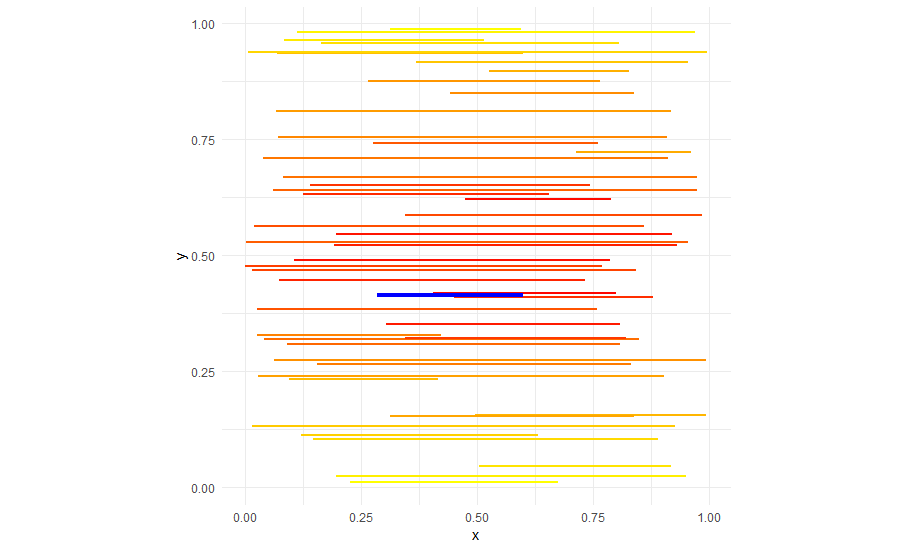}
        \caption{}
    \end{subfigure}
    \hfill
    \begin{subfigure}[b]{0.24\textwidth}
        \centering
        \includegraphics[width=\textwidth]{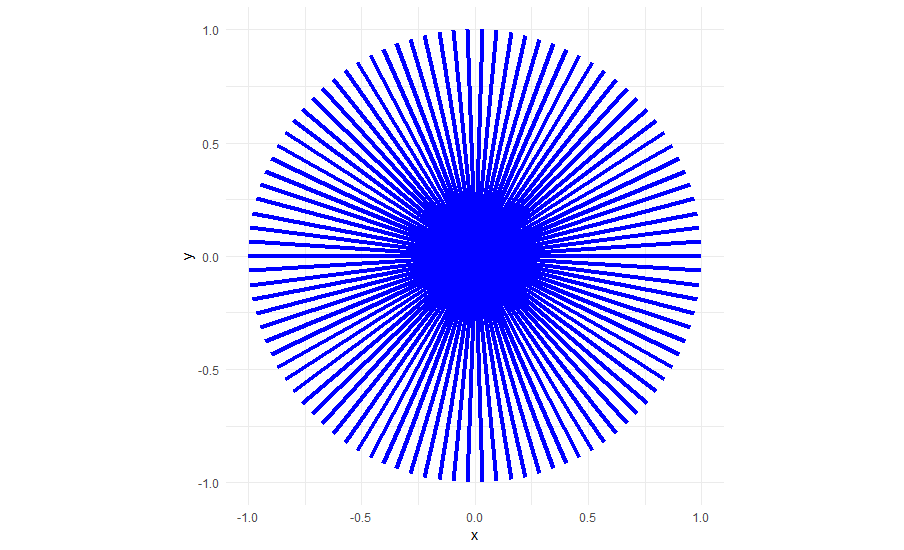}
        \caption{}
    \end{subfigure}
    \hfill
    \begin{subfigure}[b]{0.24\textwidth}
        \centering
        \includegraphics[width=\textwidth]{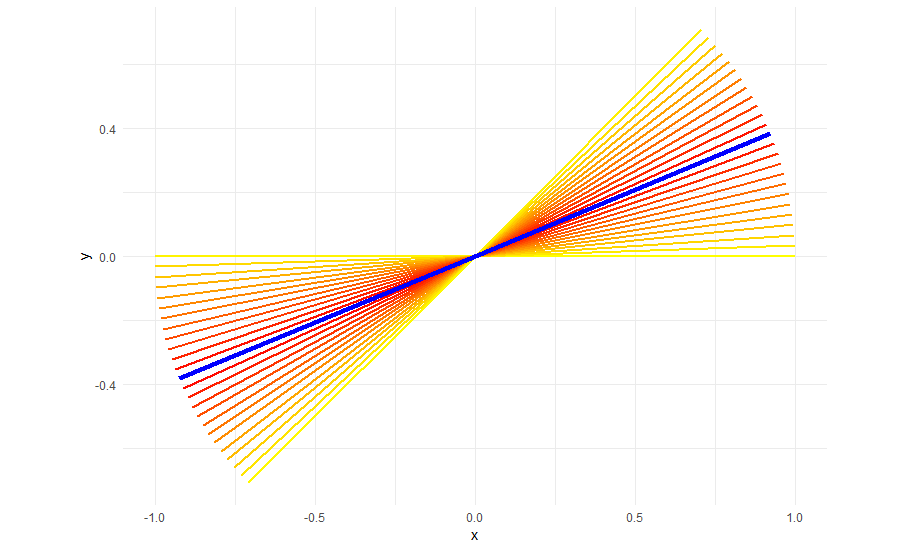}
        \caption{}
    \end{subfigure}
      \hfill
    \begin{subfigure}[b]{0.24\textwidth}
        \centering
        \includegraphics[width=\textwidth]{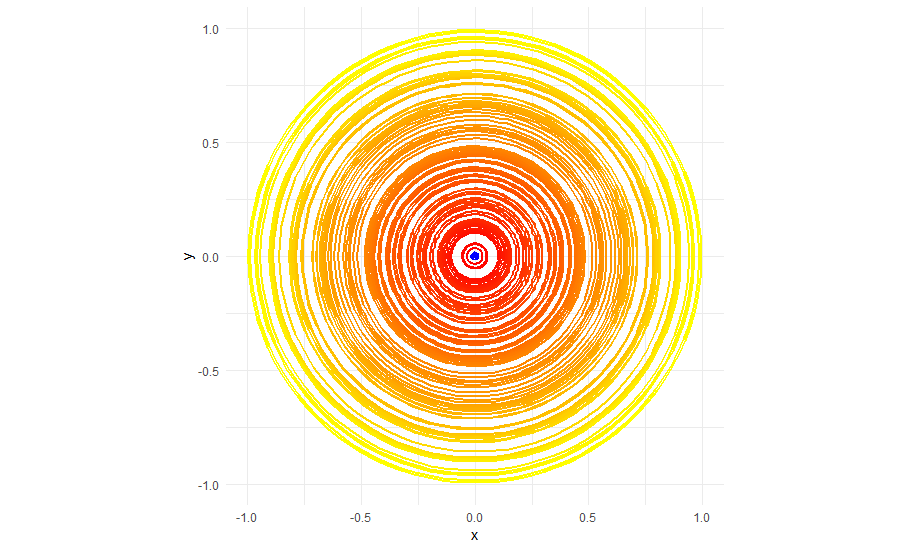}
        \caption{}
    \end{subfigure}

    \begin{subfigure}[b]{0.24\textwidth}
        \centering
        \includegraphics[width=\textwidth]{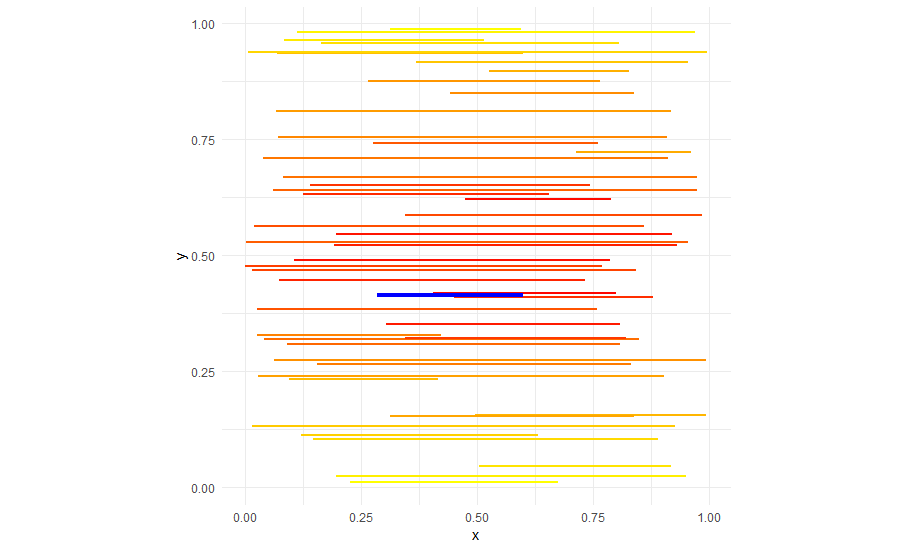}
        \caption{}
    \end{subfigure}
    \hfill
    \begin{subfigure}[b]{0.24\textwidth}
        \centering
        \includegraphics[width=\textwidth]{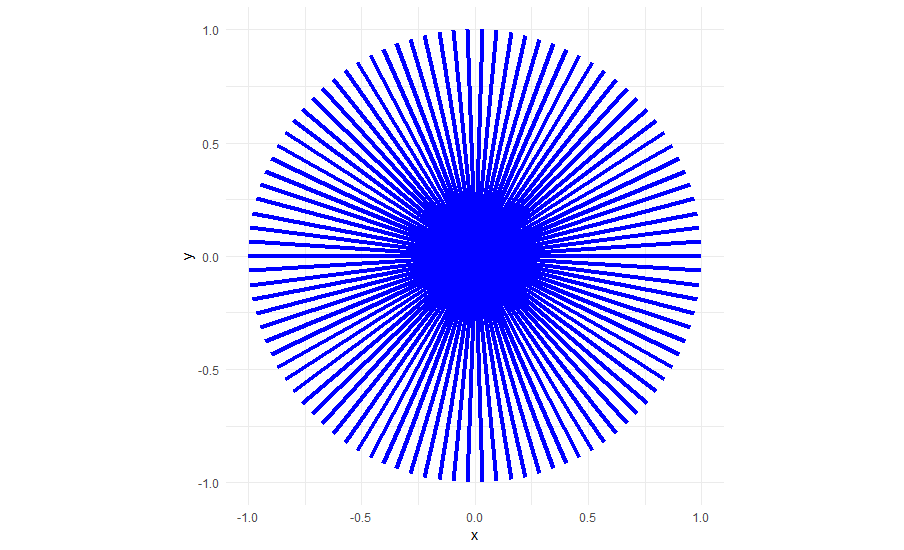}
        \caption{}
    \end{subfigure}
    \hfill
    \begin{subfigure}[b]{0.24\textwidth}
        \centering
        \includegraphics[width=\textwidth]{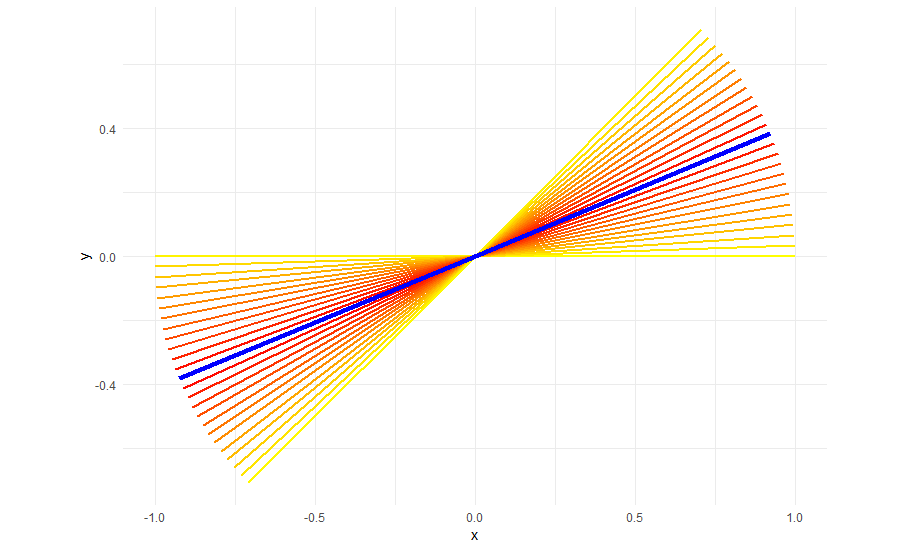}
        \caption{}
    \end{subfigure}
      \hfill
    \begin{subfigure}[b]{0.24\textwidth}
        \centering
        \includegraphics[width=\textwidth]{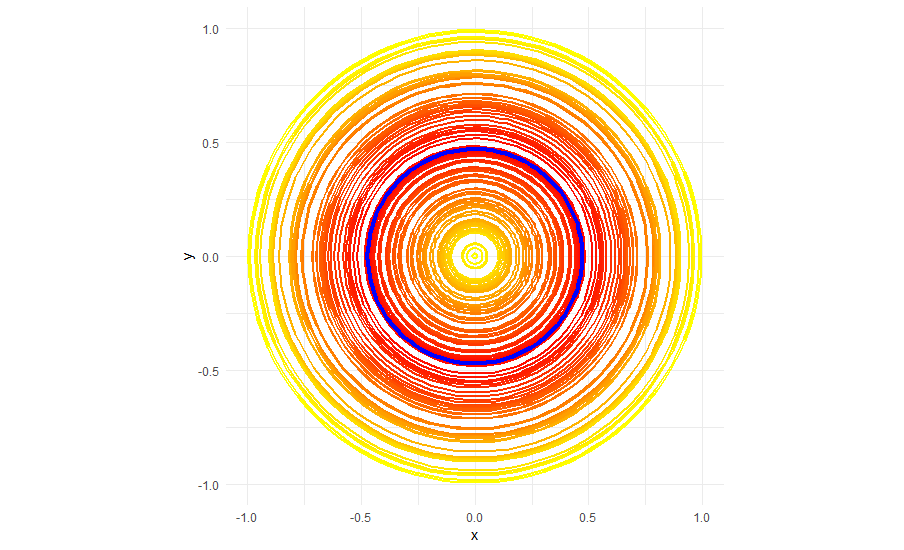}
        \caption{}
    \end{subfigure}

    \begin{subfigure}[b]{0.24\textwidth}
        \centering
        \includegraphics[width=\textwidth]{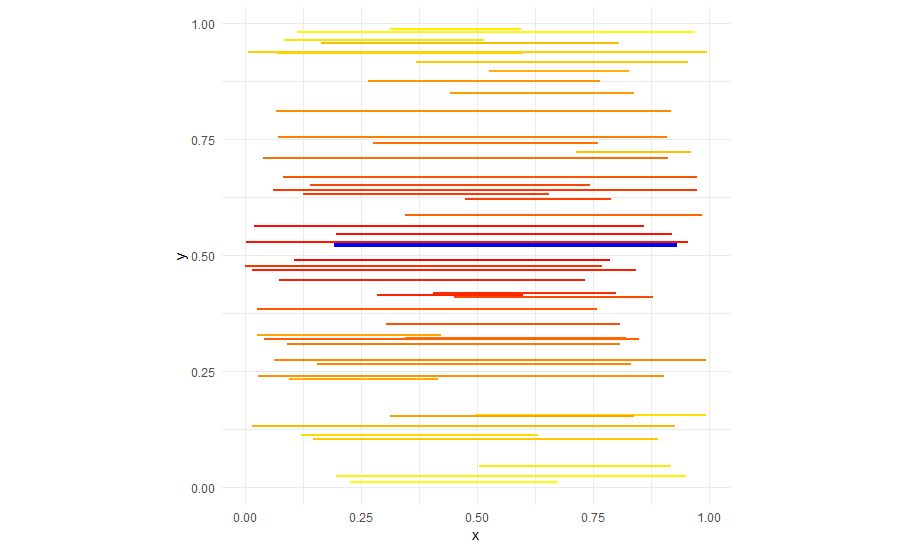}
        \caption{}
    \end{subfigure}
    \hfill
        \begin{subfigure}[b]{0.24\textwidth}
        \centering
        \includegraphics[width=\textwidth]{cbdfw_star.png}
        \caption{}
    \end{subfigure}
    \hfill
    \begin{subfigure}[b]{0.24\textwidth}
        \centering
        \includegraphics[width=\textwidth]{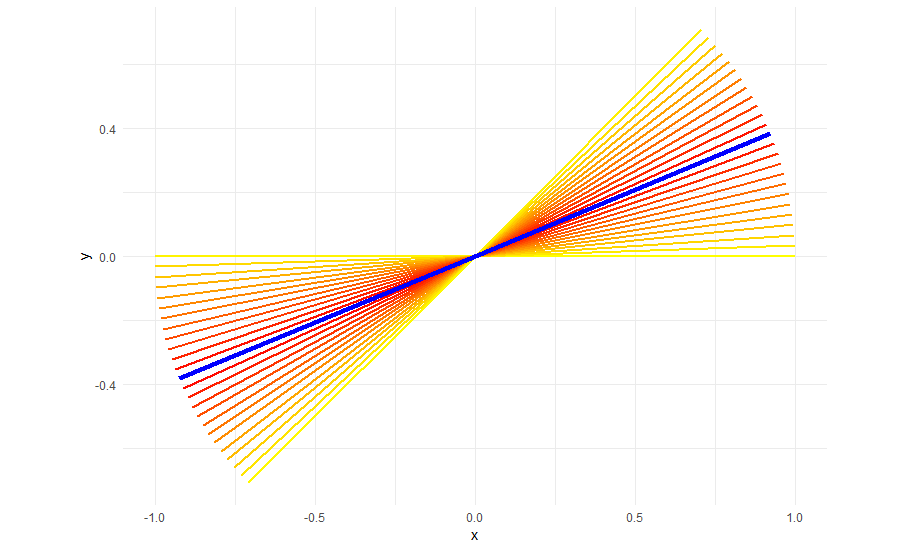}
        \caption{}
    \end{subfigure}
      \hfill
    \begin{subfigure}[b]{0.24\textwidth}
        \centering
        \includegraphics[width=\textwidth]{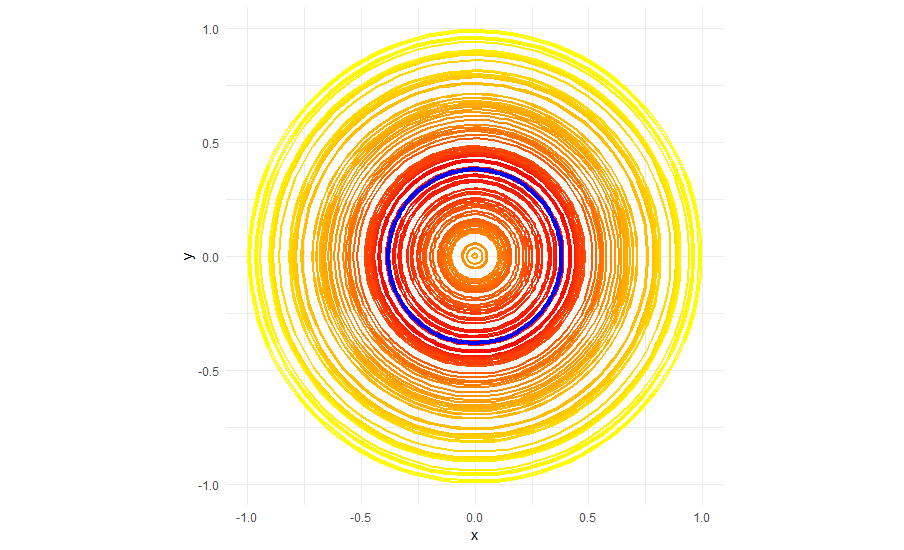}
        \caption{}
    \end{subfigure}
    \caption{Toy-example rankings under three methods: CBD (top row), FW-CBD
(middle row), and HCD (bottom row). The four columns correspond to heterogeneous
parallel segments, an equal-length star, a restricted-angle star, and concentric
circles. In each panel the deepest curve is shown in blue, and depth increases
from yellow to red.}
    \label{fig:cbdtoy}
\end{figure}
\subsection{Outlier detection}\label{sec:OD}
We construct two contamination experiments from the hand-written letters
\texttt{a} and \texttt{i} in the \texttt{R} package \texttt{mrfDepth}.
The first data set consists of $100$ randomly drawn \texttt{a} trajectories together with $10$ randomly drawn \texttt{i} trajectories; the second reverses the roles, with $100$ \texttt{i} trajectories as the majority class and $10$ \texttt{a} trajectories as the outliers.
All curves are centered and rescaled before depth computation.

Figures~\ref{fig:ai_graph}--\ref{fig:ia_order} summarize the resulting depth
rankings under CBD, length-adjusted CBD, FW-CBD, and HCD.
For the task of detecting \texttt{i} outliers among majority \texttt{a} trajectories, length-adjusted CBD is the only method that separates the outliers almost perfectly. The reason is length heterogeneity. Even after scale normalization, \texttt{a} and \texttt{i} trajectories still have systematically different arc lengths.
A typical \texttt{a} traces a loop and a descender, whereas a typical \texttt{i} is essentially a single vertical stroke, so the median \texttt{i} trajectory is noticeably shorter than the median \texttt{a} trajectory. Without a penalty, a short \texttt{i} curve embedded among majority \texttt{a} trajectories lies almost entirely inside a band generated by two \texttt{a} curves, and hence receives high CBD; the length adjustment in Section~\ref{sec:cbd_pen} corrects this. Among the remaining methods, FW-CBD recovers part of the \texttt{i} outliers, while raw CBD and HCD perform poorly.

The second task, detecting \texttt{a} outliers among majority \texttt{i} trajectories, is shown in Figures~\ref{fig:ia_graph} and \ref{fig:ia_order}.
Here FW-CBD recovers a visible portion of the \texttt{a} outliers, while the other three estimators show little separation. The FW-CBD panel in Figure~\ref{fig:ia_graph} is the only one in which the \texttt{a}-shaped outliers remain clearly visible among the lowest-depth curves (yellow \texttt{a}).
Taken together, the two contamination experiments suggest that length-adjusted CBD is most useful when contamination is mainly a length effect, and that FW-CBD is preferable when classes differ in local geometric shape rather than in global band enclosure.

\begin{figure}[h]
    \begin{subfigure}[b]{0.48\textwidth}
        \centering
        \includegraphics[width=\textwidth]{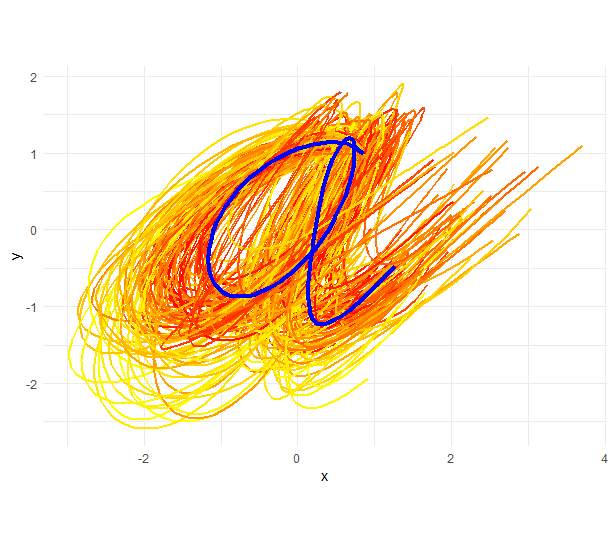}
        \caption{CBD}
    \end{subfigure}
    \hfill
    \begin{subfigure}[b]{0.48\textwidth}
        \centering
        \includegraphics[width=\textwidth]{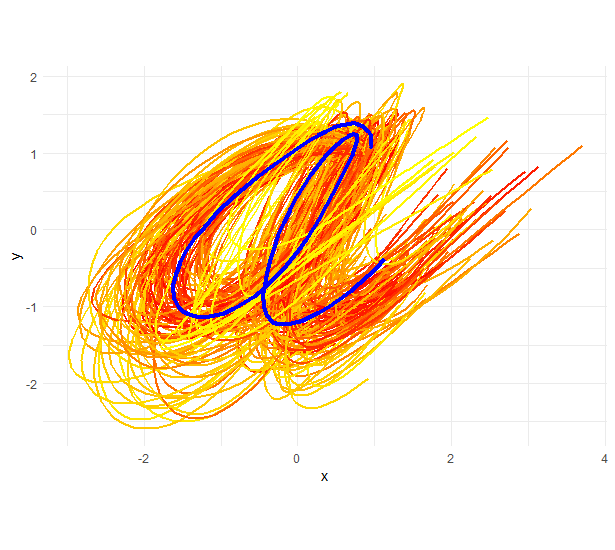}
        \caption{CBD (length adjusted)}
    \end{subfigure}

    \begin{subfigure}[b]{0.48\textwidth}
        \centering
        \includegraphics[width=\textwidth]{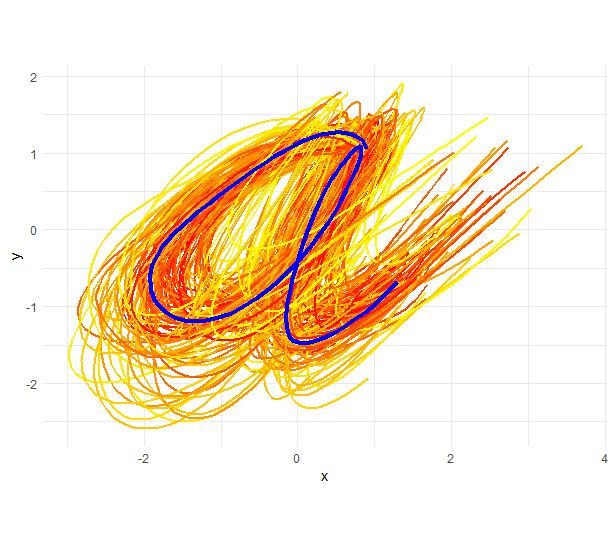}
        \caption{FW-CBD}
    \end{subfigure}
    \hfill
    \begin{subfigure}[b]{0.48\textwidth}
        \centering
        \includegraphics[width=\textwidth]{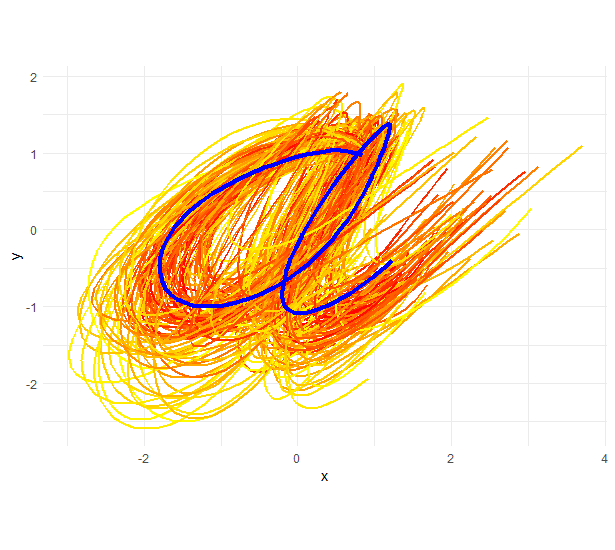}
        \caption{HCD}
    \end{subfigure}
    \caption{Depth-colored curve plots for the first contamination experiment
    ($100$ \texttt{a} trajectories and $10$ \texttt{i} trajectories).}
    \label{fig:ai_graph}
\end{figure}

\begin{figure}[h]
    \begin{subfigure}[b]{0.48\textwidth}
        \centering
        \includegraphics[width=\textwidth]{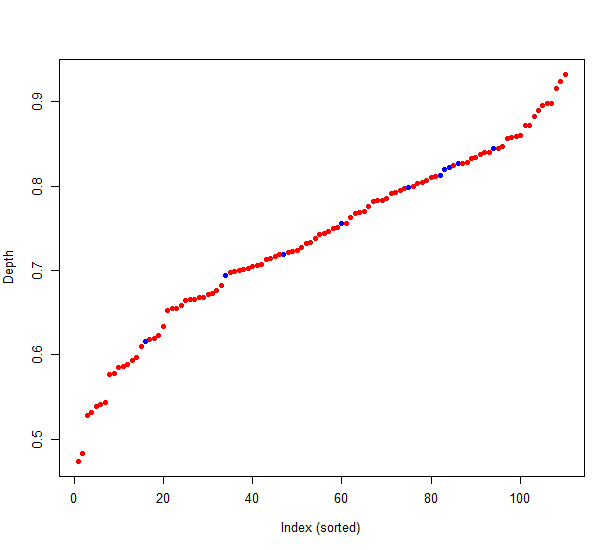}
        \caption{CBD}
    \end{subfigure}
    \hfill
    \begin{subfigure}[b]{0.48\textwidth}
        \centering
        \includegraphics[width=\textwidth]{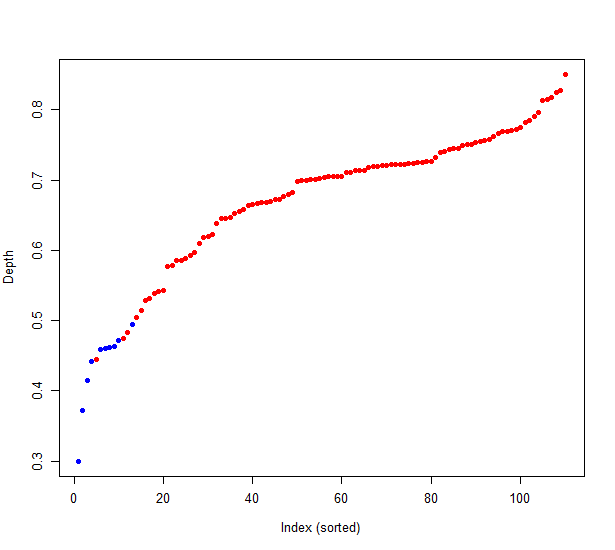}
        \caption{CBD (length adjusted)}
    \end{subfigure}

    \begin{subfigure}[b]{0.48\textwidth}
        \centering
        \includegraphics[width=\textwidth]{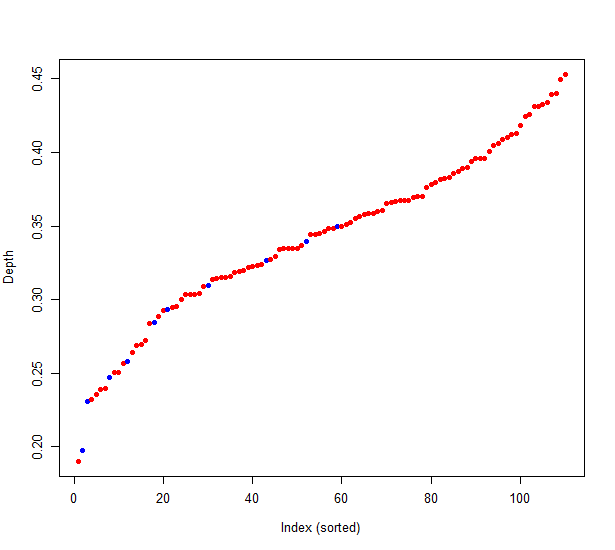}
        \caption{FW-CBD}
    \end{subfigure}
    \hfill
    \begin{subfigure}[b]{0.48\textwidth}
        \centering
        \includegraphics[width=\textwidth]{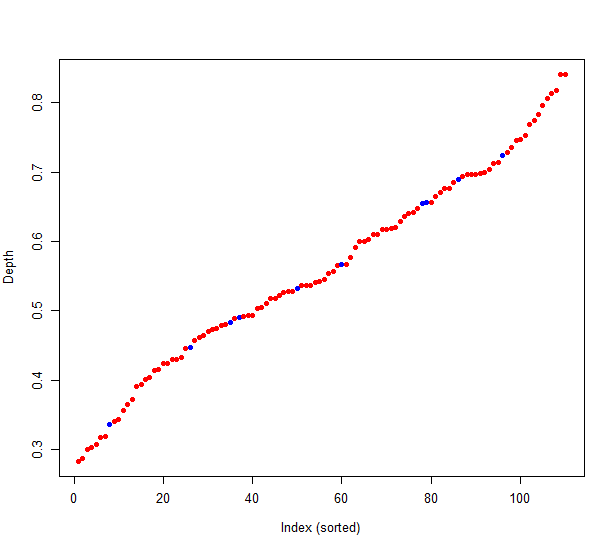}
        \caption{HCD}
    \end{subfigure}

    \caption{Ordered depth values for the first contamination experiment,
    where blue points correspond to the \texttt{i} outliers.}
    \label{fig:ai_order}
\end{figure}

\begin{figure}[h]
    \begin{subfigure}[b]{0.48\textwidth}
        \centering
        \includegraphics[width=\textwidth]{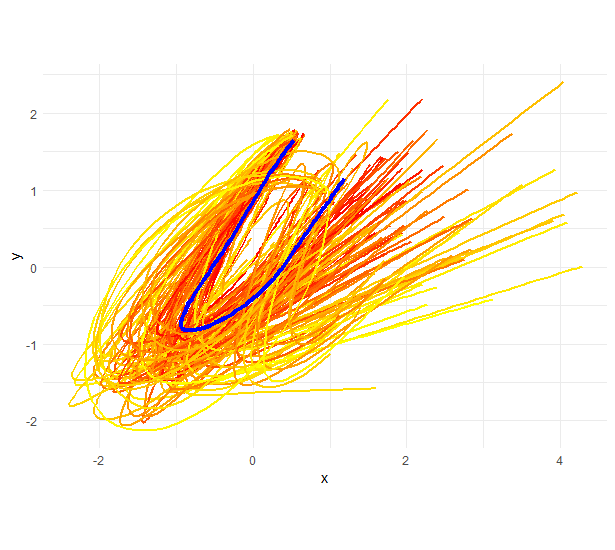}
        \caption{CBD}
    \end{subfigure}
    \hfill
    \begin{subfigure}[b]{0.48\textwidth}
        \centering
        \includegraphics[width=\textwidth]{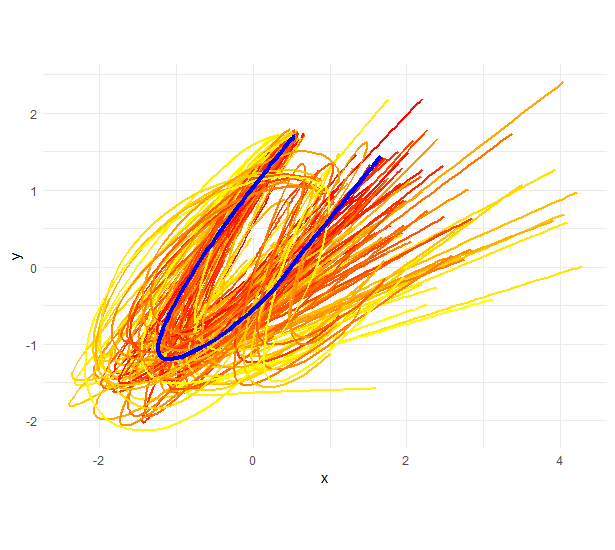}
        \caption{CBD (length adjusted)}
    \end{subfigure}

    \begin{subfigure}[b]{0.48\textwidth}
        \centering
        \includegraphics[width=\textwidth]{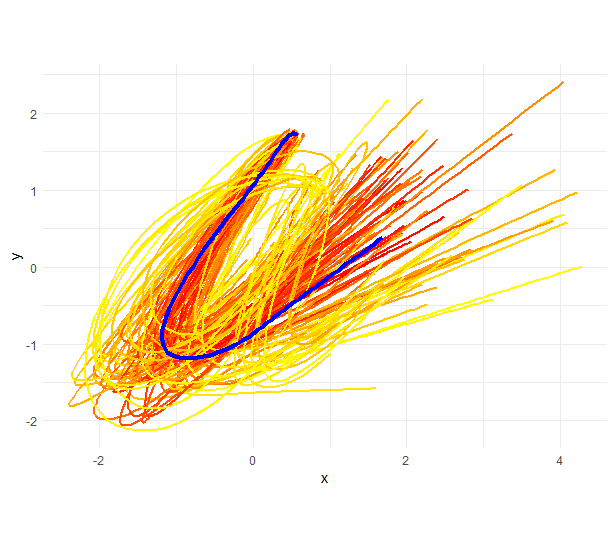}
        \caption{FW-CBD}
    \end{subfigure}
    \hfill
    \begin{subfigure}[b]{0.48\textwidth}
        \centering
        \includegraphics[width=\textwidth]{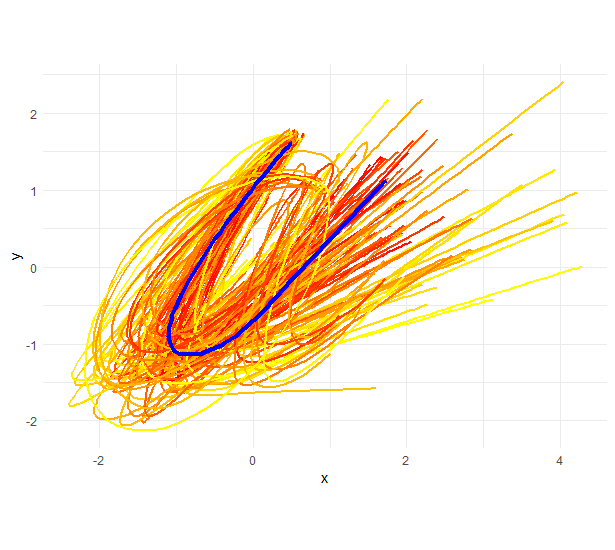}
        \caption{HCD}
    \end{subfigure}

    \caption{Depth-colored curve plots for the second contamination experiment
    ($100$ \texttt{i} trajectories and $10$ \texttt{a} trajectories).}
    \label{fig:ia_graph}
\end{figure}

\begin{figure}[h]
    \begin{subfigure}[b]{0.48\textwidth}
        \centering
        \includegraphics[width=\textwidth]{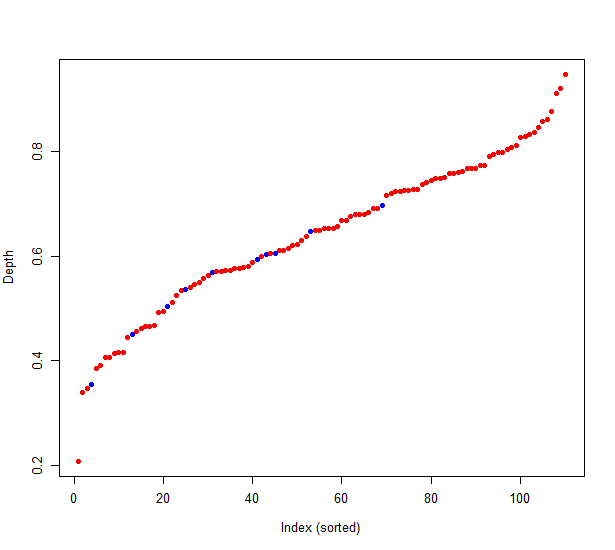}
        \caption{CBD}
    \end{subfigure}
    \hfill
    \begin{subfigure}[b]{0.48\textwidth}
        \centering
        \includegraphics[width=\textwidth]{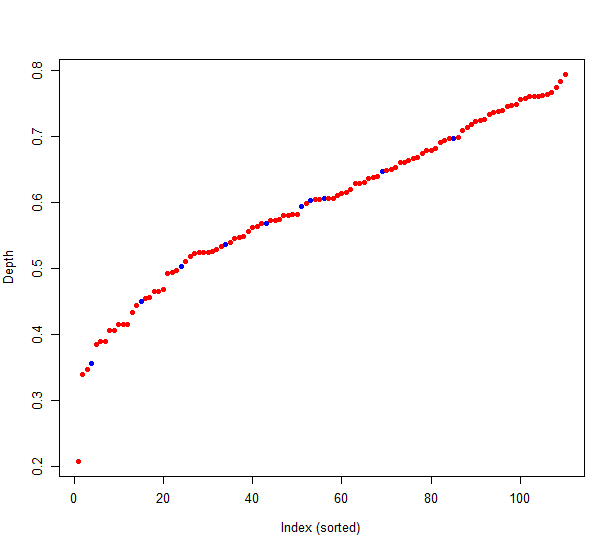}
        \caption{CBD (length adjusted)}
    \end{subfigure}

    \begin{subfigure}[b]{0.48\textwidth}
        \centering
        \includegraphics[width=\textwidth]{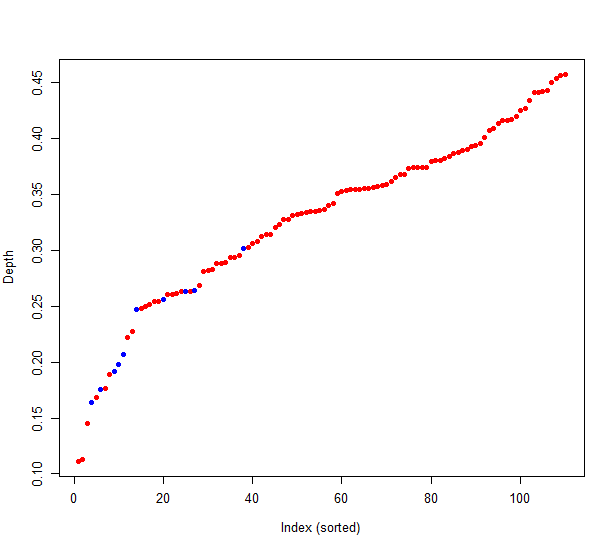}
        \caption{FW-CBD}
    \end{subfigure}
    \hfill
    \begin{subfigure}[b]{0.48\textwidth}
        \centering
        \includegraphics[width=\textwidth]{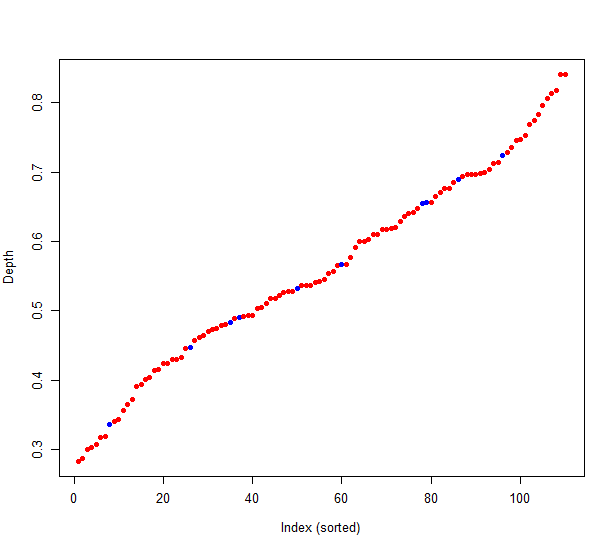}
        \caption{HCD}
    \end{subfigure}
    \caption{Ordered depth values for the second contamination experiment,
    where blue points correspond to the \texttt{a} outliers.}
    \label{fig:ia_order}
\end{figure}
\clearpage
\subsection{Classification}
We next study whether the same geometric differences among the depth notions
translate into classification performance.
For both data sets, we use the depth values with respect to the two training
classes as a two-dimensional depth representation and apply the DD$\alpha$
classifier~\citep{pokotylo2019depth} to the resulting points.
The first task uses the hand-written letters \texttt{a} and \texttt{i} as in Section~\ref{sec:OD}; the second uses curve representations of the MNIST
digits \texttt{1} and \texttt{7}, using the pre-processed data of~\citet{de2021depth}.

Table~\ref{tab:acc_ai_depth_methods} reports the results for the \texttt{a}/\texttt{i} task.
Length adjustment helps considerably: mean accuracy rises from $0.749$ for raw CBD to $0.955$ for length-adjusted CBD, a gain of about $0.21$ per split.
FW-CBD is second, with HCD and raw CBD clearly worse.
This matches the outlier-detection results above and the length-heterogeneity explanation. Even after scale normalization, \texttt{a} and \texttt{i} still differ systematically in arc length, so unpenalized CBD assigns overly large depth to short \texttt{i}-type trajectories embedded in the \texttt{a} class.

The MNIST \texttt{1}/\texttt{7} task in Table~\ref{tab:acc_mnist17_depth_methods} shows a different pattern.
All four methods perform well, and the gain from length adjustment is much smaller. The arc-length distributions of the two digit classes are closer, so the length penalty has less to do. Length-adjusted CBD still attains the best mean accuracy, with FW-CBD again second.
\begin{table}[th]
\centering
\caption{Classification accuracy (mean $\pm$ SD over $10$ random splits) for
handwritten letters \texttt{a} and \texttt{i}. In each split, $100$ curves per
class are used for training and the remaining curves ($71$ \texttt{a} and
$74$ \texttt{i}) form the test set. Ranks are computed within each split across
the four methods ($1=\text{best}$), and reported as mean ranks over splits.}
\label{tab:acc_ai_depth_methods}
\begin{tabular}{lcc}
\hline
\textbf{Method} & \textbf{Accuracy (mean $\pm$ SD)} & \textbf{Mean rank} \\
\hline
CBD (raw) & $0.749 \pm 0.046$ & 3.7 \\
CBD (length-adjusted) & $0.955 \pm 0.032$ & 1.0 \\
FW-CBD & $0.891 \pm 0.019$ & 2.0 \\
HCD & $0.776 \pm 0.044$ & 3.3 \\
\hline
\end{tabular}
\end{table}

\begin{table}[th]
\centering
\caption{Classification accuracy (mean $\pm$ SD over $10$ random splits) for MNIST
digit curves \texttt{1} and \texttt{7}. In each split, $50$ curves per class are
used for training and $50$ per class form the test set. Ranks are computed within
each split across the four methods ($1=\text{best}$), and reported as mean ranks
over splits.}
\label{tab:acc_mnist17_depth_methods}
\begin{tabular}{lcc}
\hline
\textbf{Method} & \textbf{Accuracy (mean $\pm$ SD)} & \textbf{Mean rank} \\
\hline
CBD (raw) & $0.922 \pm 0.041$ & 3.1 \\
CBD (length-adjusted) & $0.975 \pm 0.018$ & 1.1 \\
FW-CBD & $0.955 \pm 0.013$ & 2.2 \\
HCD & $0.938 \pm 0.010$ & 2.9 \\
\hline
\end{tabular}
\end{table}
\subsection{Online signature screening on MOBISIG}
\label{sec:mobisig}

We next consider a one-class curve-screening task on the \texttt{MOBISIG} online-signature dataset.
Each signature is treated as a planar trajectory, and the goal is to decide whether a new signature is geometrically consistent with a small reference set of genuine signatures from the same user.
A depth-based approach fits this problem directly because low depth against the genuine reference set flags a potential forgery.

The dataset contains $83$ users, each with $45$ genuine signatures and
$20$ skilled forgeries.
We retain only the two-dimensional $(x,y)$ trajectory, remove non-finite rows
and consecutive duplicates, center each trajectory, and isotropically rescale
it so that the longer side of its bounding box equals $1$.
Depth is then computed on the arc-length-resampled representation.

We evaluate two one-class enrollment protocols. Protocol~A uses a single
random seed for each user and takes
$n_{\mathrm{ref}}\in\{10,20,30\}$ genuine signatures as the reference set; the
remaining genuine signatures and all skilled forgeries targeting the same user
form the test set. In Protocol~B, the first $15$ genuine signatures are used as the reference set, and the remaining genuine signatures together with all forgeries form the test set, thereby probing robustness to session shift. For every test signature
we compute CBD, FW-CBD, the length-adjusted versions of both of them and HCD.
Performance is summarized in Table~\ref{tab:mobisig_results} by per-user area
under curve (AUC) and equal error rate (EER), aggregated across users by median
and interquartile range (IQR); the Protocol~A and Protocol~B per-user
distributions appear in Figure~\ref{fig:mobisig_box}.

The results give a consistent geometric message.  FW-CBD substantially improves
over raw CBD, and the length-adjusted FW-CBD score is the strongest variant in
all Protocol~A enrollment regimes and in the Protocol~B session-shift setting.
The length adjustment also helps CBD relative to raw CBD, but the gap between length adjusted CBD and length
adjusted FW-CBD remains large.  This supports the interpretation that the
narrower FW band captures a stricter notion of signature-shape agreement, while
the length penalty corrects a separate nuisance effect caused by heterogeneous
trajectory lengths.

\begin{table}[t!]
\centering
\small
\caption{MOBISIG single-split signature-screening results. Entries are median $\pm$ IQR across the $83$ users. Mean ranks are computed within each user across the five methods (1 = best) and then averaged across users.}
\label{tab:mobisig_results}
\begin{tabular}{llccc}
\hline
\textbf{Enrollment} & \textbf{Method} & \textbf{AUC (median $\pm$ IQR)} & \textbf{EER (median $\pm$ IQR)} & \textbf{Mean rank} \\
\hline
\multicolumn{5}{l}{\emph{Protocol A: random enrollment, $n_{\mathrm{ref}}=10$}} \\
 & CBD & $0.744 \pm 0.321$ & $0.307 \pm 0.234$ & 4.181 \\
 & CBD (length-adjusted) & $0.801 \pm 0.213$ & $0.293 \pm 0.146$ & 3.476 \\
 & FW-CBD & $0.909 \pm 0.175$ & $0.161 \pm 0.146$ & 2.012 \\
 & FW-CBD (length-adjusted) & $\mathbf{0.931 \pm 0.120}$ & $\mathbf{0.146 \pm 0.139}$ & \textbf{1.590} \\
 & HCD & $0.763 \pm 0.247$ & $0.307 \pm 0.234$ & 3.741 \\
\hline
\multicolumn{5}{l}{\emph{Protocol A: random enrollment, $n_{\mathrm{ref}}=20$}} \\
 & CBD & $0.718 \pm 0.353$ & $0.310 \pm 0.300$ & 4.265 \\
 & CBD (length-adjusted) & $0.806 \pm 0.206$ & $0.245 \pm 0.200$ & 3.542 \\
 & FW-CBD & $0.918 \pm 0.150$ & $0.155 \pm 0.135$ & 1.946 \\
 & FW-CBD (length-adjusted) & $\mathbf{0.938 \pm 0.111}$ & $\mathbf{0.155 \pm 0.110}$ & \textbf{1.554} \\
 & HCD & $0.782 \pm 0.299$ & $0.290 \pm 0.222$ & 3.693 \\
\hline
\multicolumn{5}{l}{\emph{Protocol A: random enrollment, $n_{\mathrm{ref}}=30$}} \\
 & CBD & $0.723 \pm 0.337$ & $0.342 \pm 0.317$ & 4.301 \\
 & CBD (length-adjusted) & $0.810 \pm 0.213$ & $0.258 \pm 0.171$ & 3.542 \\
 & FW-CBD & $0.923 \pm 0.135$ & $0.142 \pm 0.142$ & 1.952 \\
 & FW-CBD (length-adjusted) & $\mathbf{0.943 \pm 0.088}$ & $\mathbf{0.142 \pm 0.142}$ & \textbf{1.518} \\
 & HCD & $0.803 \pm 0.270$ & $0.317 \pm 0.200$ & 3.687 \\
\hline
\multicolumn{5}{l}{\emph{Protocol B: session-shift enrollment,$n_{\mathrm{ref}}=15$}} \\
 & CBD & $0.683 \pm 0.432$ & $0.400 \pm 0.338$ & 3.964 \\
 & CBD (length-adjusted) & $0.722 \pm 0.314$ & $0.358 \pm 0.279$ & 3.361 \\
 & FW-CBD & $0.837 \pm 0.246$ & $0.258 \pm 0.200$ & 2.157 \\
 & FW-CBD (length-adjusted) & $\mathbf{0.872 \pm 0.183}$ & $\mathbf{0.242 \pm 0.183}$ & \textbf{1.934} \\
 & HCD & $0.683 \pm 0.326$ & $0.358 \pm 0.258$ & 3.584 \\
\hline
\end{tabular}
\end{table}

\begin{figure}[t!]
    \centering
    \begin{subfigure}[b]{\textwidth}
        \centering
        \includegraphics[width=\textwidth]{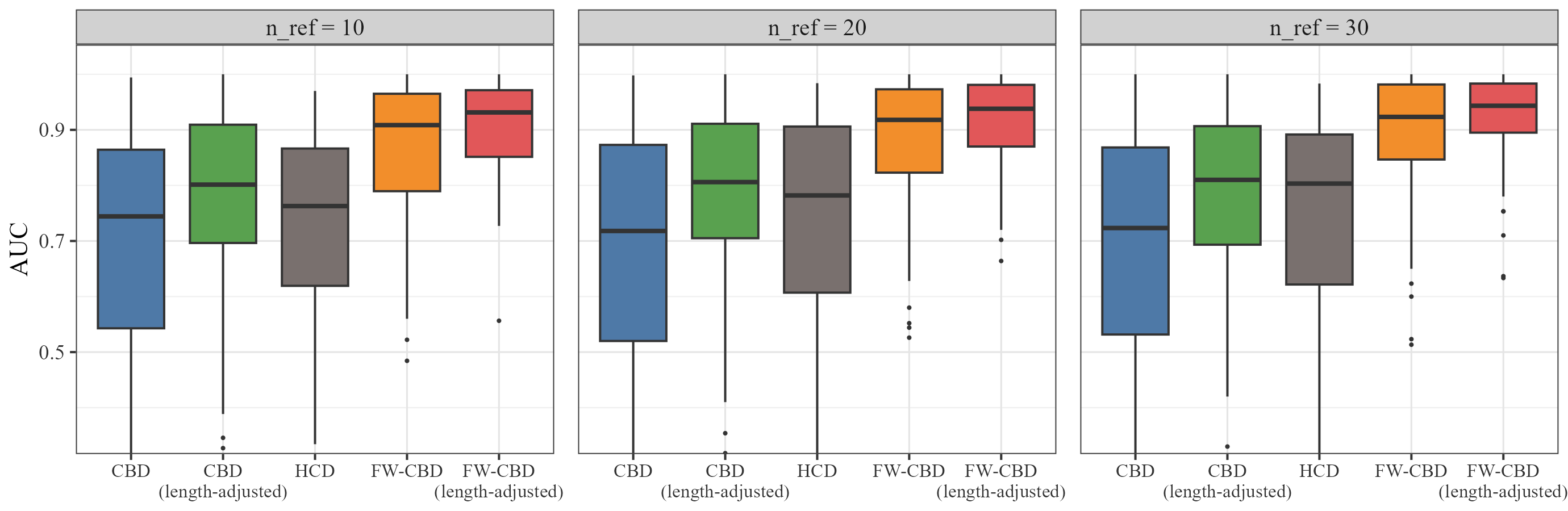}
        \caption{Protocol~A AUC.}
    \end{subfigure}

    \vspace{0.5em}
    \begin{subfigure}[b]{\textwidth}
        \centering
        \includegraphics[width=\textwidth]{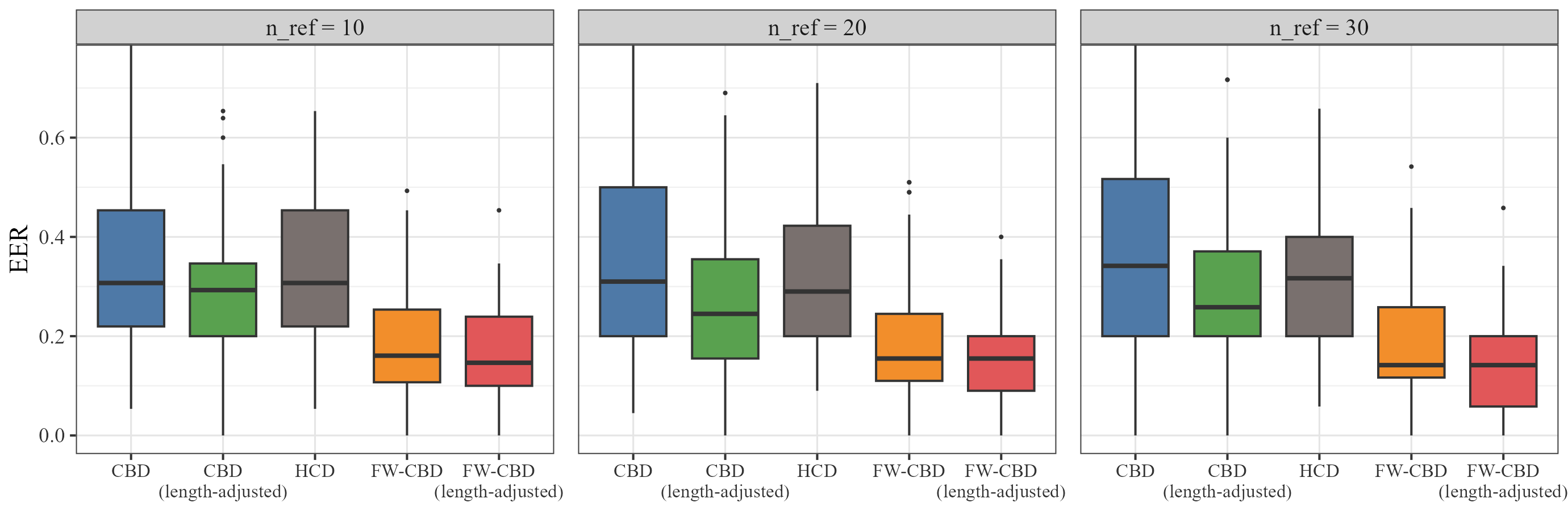}
        \caption{Protocol~A EER.}
    \end{subfigure}

    \vspace{0.5em}
    \begin{subfigure}[b]{0.49\textwidth}
        \centering
        \includegraphics[width=\textwidth]{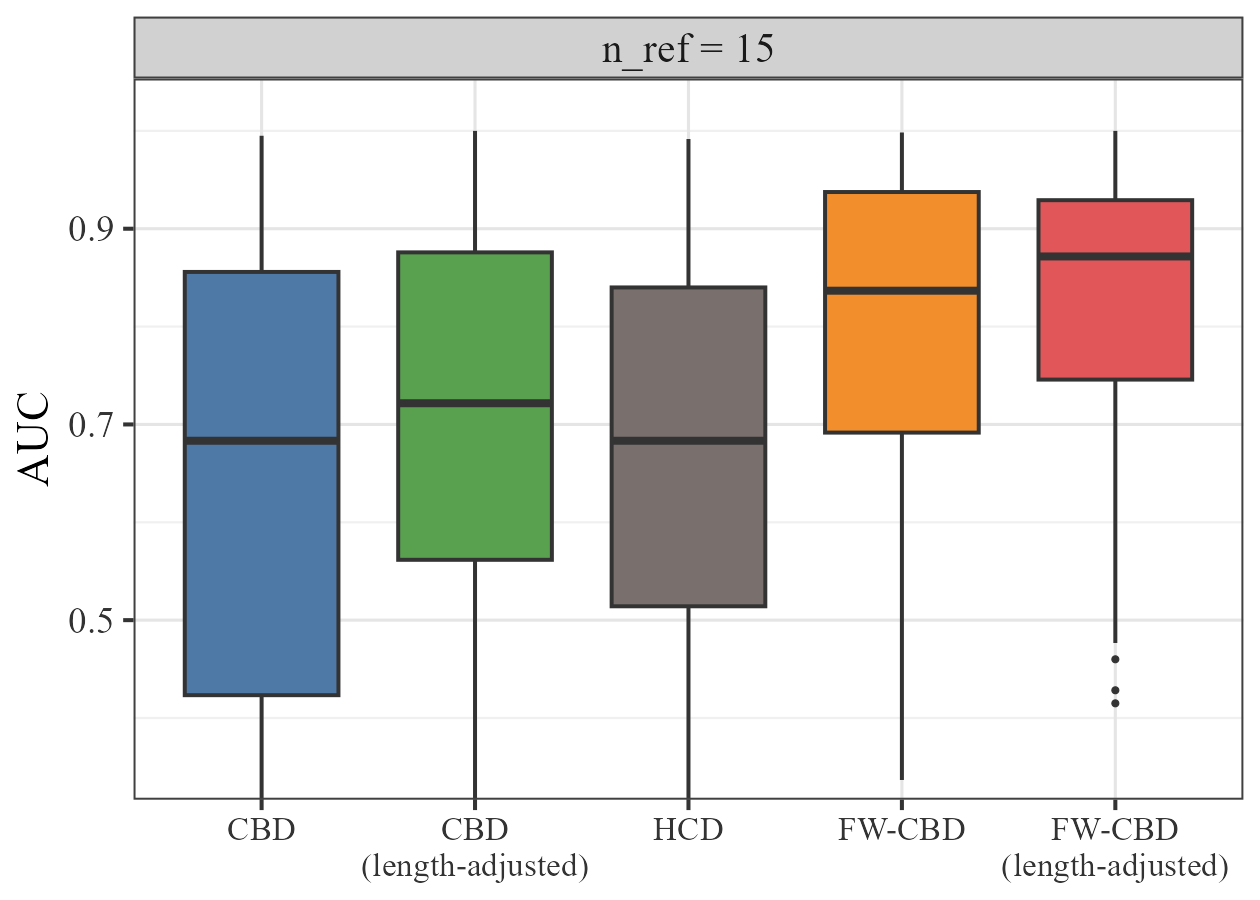}
        \caption{Protocol~B AUC.}
    \end{subfigure}
    \hfill
    \begin{subfigure}[b]{0.49\textwidth}
        \centering
        \includegraphics[width=\textwidth]{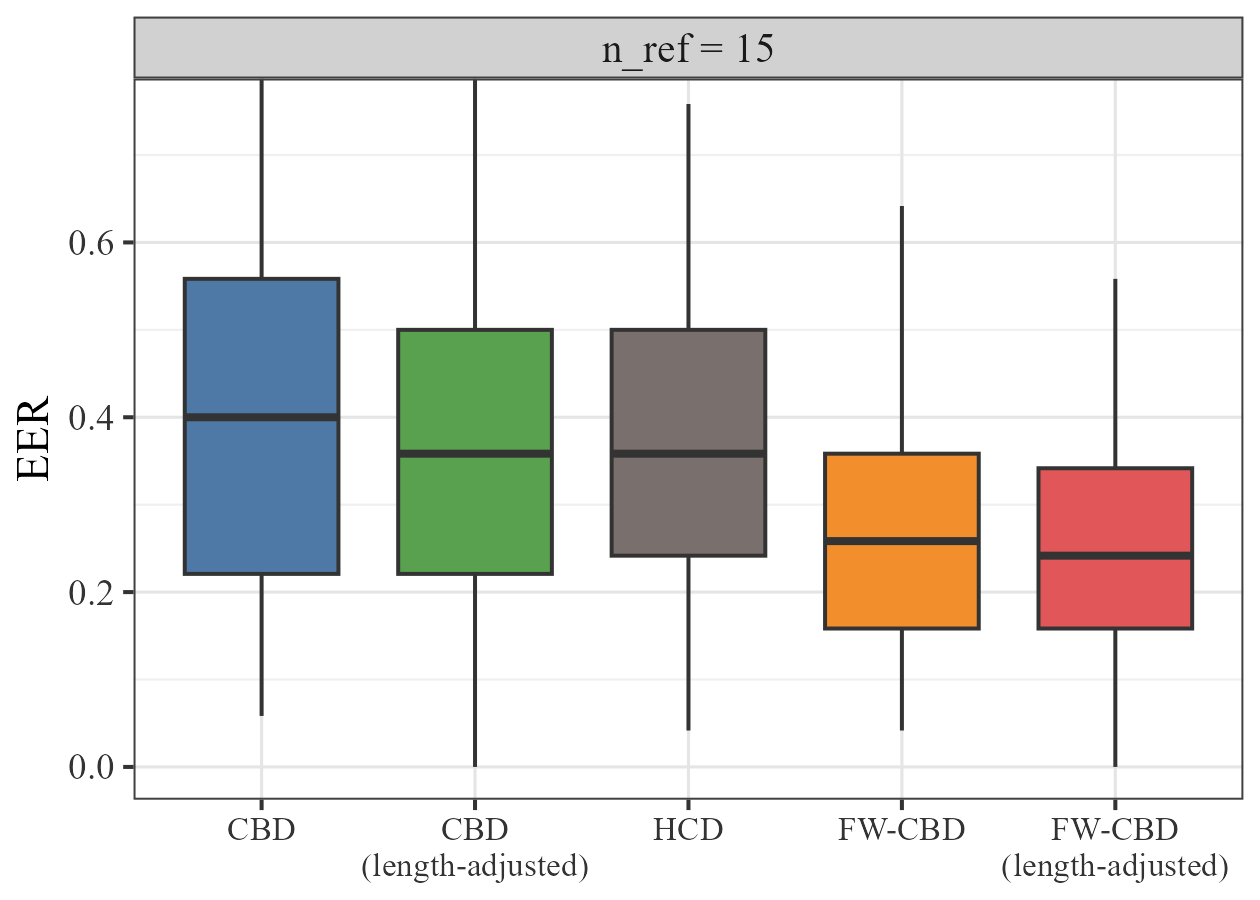}
        \caption{Protocol~B EER.}
    \end{subfigure}
    \caption{Per-user AUC and EER distributions for the five depth scores on MOBISIG. Protocol~A uses random enrollment with $n_{\mathrm{ref}}\in\{10,20,30\}$; Protocol~B uses the first-session genuine signatures as the reference set. Each box aggregates $83$ per-user values.}
    \label{fig:mobisig_box}
\end{figure}

\FloatBarrier
\subsection{Clustering from decomposed band support}
\label{sec:uci_contribution_clustering}
The scalar CBD of a target curve is an average of its in-band support over all bands generated by pairs of sample curves.  To use CBD values in clustering tasks, we assign the support contributed by each band back to the two curves that generate the band, thereby producing for each target curve a contribution profile over the observed curves.  These profiles are then symmetrized into an affinity matrix and clustered by Ward's minimum-variance agglomerative hierarchical clustering.  Unless otherwise stated, all hierarchical clustering results in this subsection use this same agglomerative criterion.

We first define the auxiliary curve distance used for attribution.  Let $I=[0,1]$ for open curves and $I=\bbS^1$ for closed curves.  For two curve classes $A,B\in\Cset$ of the same type, define the Fr\'echet-type distance
\begin{equation}
\label{eq:clust_curve_distance}
\delta(A,B)
:=
\inf_{\alpha\in\mathcal G(A),\,\beta\in\mathcal G(B)}
\inf_{\phi\in\mathcal H_I}
\sup_{t\in I}
\|\alpha(t)-\beta(\phi(t))\|_2,
\end{equation}
where $\mathcal G(\cdot)$ is the family of admissible normalized representatives from Definition~\ref{def:normapara}, and $\mathcal H_I$ is the set of orientation-preserving homeomorphisms of $I$.  In the numerical implementation, $\delta$ is computed by \texttt{dist.curves} from the \texttt{curveDepth} package~\citep{curveDepth2025}.

Let $r\in\{\mathrm{CBD},\mathrm{FW}\}$ denote the band type, with $\Rband_{\mathrm{CBD}}=\Rband$ and $\Rband_{\mathrm{FW}}=R_{\mathrm{FW}}$.  For a target curve $C_k$ and a band generated by $C_a$ and $C_c$, $a<c$, define the band-support contribution
$
q_{k,ac}^{(r)}
:=
\mu_{C_k}\!\left(\Rband_r(C_a,C_c)\right).
$
Thus the integral depth of $C_k$ with respect to the sample can be written as
\begin{equation}
\label{eq:clust_depth_decomp}
D_{\mathrm{int}}^{(r)}(C_k\mid\{C_i\}_{i=1}^n)
=
\binom{n}{2}^{-1}
\sum_{1\le a<c\le n}q_{k,ac}^{(r)}.
\end{equation}
Equation~\eqref{eq:clust_depth_decomp} shows that CBD first forms band-level supports and then averages them.  The clustering construction below keeps the band-level supports and attributes them to source curves.

For a fixed pair $(a,c)$, the support $q_{k,ac}^{(r)}$ is split between $C_a$ and $C_c$ according to the relative position of $C_k$ on the distance chord joining them.  Let $\delta_{uv}:=\delta(C_u,C_v)$ and define
$
\operatorname{clip}_{[0,1]}(x):=\min\{1,\max\{0,x\}\}.
$
The weight assigned to $C_c$ is
\begin{equation}
\label{eq:clust_weight_c}
\omega_{k\to c}^{(a,c)}
:=
\begin{cases}
\displaystyle
\operatorname{clip}_{[0,1]}\!\left(
\frac{\delta_{ka}^2+\delta_{ac}^2-\delta_{kc}^2}{2\delta_{ac}^2}
\right), & \delta_{ac}>0,\\[1.2em]
1/2, & \delta_{ac}=0,
\end{cases}
\end{equation}
and $\omega_{k\to a}^{(a,c)}:=1-\omega_{k\to c}^{(a,c)}$.  Hence a target curve closer to $C_c$ assigns more of the band support to $C_c$, and analogously for $C_a$.

The directed contribution matrix $\mathsf B^{(r)}=(\mathsf B_{kj}^{(r)})$ is defined by
\begin{equation}
\label{eq:clust_directed_contribution}
\mathsf B_{kj}^{(r)}
:=
\sum_{1\le a<c\le n}
q_{k,ac}^{(r)}
\left[
\one\{j=a,\ j\ne k\}\omega_{k\to a}^{(a,c)}
+
\one\{j=c,\ j\ne k\}\omega_{k\to c}^{(a,c)}
\right].
\end{equation}
The condition $j\ne k$ removes self-contribution, since the profile is used to measure the relation of $C_k$ to other curves rather than to explain its self-support.  We row-normalize $\mathsf B^{(r)}$ to obtain a directed contribution profile
\begin{equation}
\label{eq:clust_profile}
\mathsf P_{kj}^{(r)}
:=
\begin{cases}
\displaystyle
\mathsf B_{kj}^{(r)}\Big/\sum_{\ell\ne k}\mathsf B_{k\ell}^{(r)},
& \sum_{\ell\ne k}\mathsf B_{k\ell}^{(r)}>0,\\[0.8em]
0, & \sum_{\ell\ne k}\mathsf B_{k\ell}^{(r)}=0.
\end{cases}
\end{equation}
Finally, the directed profiles are symmetrized into an affinity matrix
\begin{equation}
\label{eq:clust_affinity}
\mathsf A_{ij}^{(r)}
:=
\frac{\mathsf P_{ij}^{(r)}+\mathsf P_{ji}^{(r)}}{2},
\qquad
\mathsf A_{ii}^{(r)}:=0.
\end{equation}
After rescaling $\mathsf A^{(r)}$ by its largest off-diagonal entry, we use
\begin{equation}
\label{eq:clust_dissimilarity}
\mathsf M_{ij}^{(r)}:=1-\mathsf A_{ij}^{(r)},\quad i\ne j,
\qquad
\mathsf M_{ii}^{(r)}:=0
\end{equation}
as the dissimilarity matrix for hierarchical clustering.  The full algorithm is summarized in Algorithm~\ref{alg:contribution_affinity_clustering}.

\begin{algorithm}[ht]
\caption{Contribution-affinity clustering from decomposed CBD}
\label{alg:contribution_affinity_clustering}
\begin{algorithmic}[1]
\Function{ContributionAffinityClustering}{$C_1,\ldots,C_n$, $K$, $r$}
\State Compute $\delta_{ij}=\delta(C_i,C_j)$ using \eqref{eq:clust_curve_distance}.
\State Compute the band supports $q_{k,ac}^{(r)}$ for all $k$ and all $a<c$.
\State Form $\mathsf B^{(r)}$ using \eqref{eq:clust_directed_contribution}.
\State Row-normalize $\mathsf B^{(r)}$ to obtain $\mathsf P^{(r)}$ using \eqref{eq:clust_profile}.
\State Symmetrize $\mathsf P^{(r)}$ to obtain $\mathsf A^{(r)}$ using \eqref{eq:clust_affinity}.
\State Rescale $\mathsf A^{(r)}$ by its largest off-diagonal entry and form $\mathsf M^{(r)}$ using \eqref{eq:clust_dissimilarity}.
\State Apply hierarchical clustering to $\mathsf M^{(r)}$ and cut the dendrogram into $K$ clusters.
\State \Return cluster labels $\hat z_1,\ldots,\hat z_n$.
\EndFunction
\end{algorithmic}
\end{algorithm}
\begin{table}[t!]
\centering
\small
\caption{Clustering comparisons on UCI Character Trajectories. Accuracy is computed after optimal label matching. The CBD and FW-CBD columns use Algorithm~\ref{alg:contribution_affinity_clustering} with the global band $\Rband$ and the fast-walk band $R_{\mathrm{FW}}$, respectively.}
\label{tab:uci_depth_clustering}
\begin{tabular}{lccc}
\toprule
\textbf{Letters} & \textbf{Distance} & \textbf{CBD affinity} & \textbf{FW-CBD affinity} \\
\midrule
\texttt{aho} & 0.511 & 0.844 & 0.844 \\
\texttt{mwz} & 0.533 & 0.933 & 0.956 \\
\texttt{nor} & 0.511 & 0.867 & 0.911 \\
\texttt{uvw} & 0.867 & 0.933 & 0.889 \\
\texttt{bgy} & 0.889 & 0.867 & 0.889 \\
\midrule
\text{Mean} & 0.662 & 0.889 & 0.898 \\
\bottomrule
\end{tabular}
\end{table}
Table~\ref{tab:uci_depth_clustering} reports a small analysis on UCI Character Trajectories.  We consider five three-letter subsets designed to be nontrivial for clustering: \texttt{aho}, \texttt{mwz}, \texttt{nor}, \texttt{uvw}, and \texttt{bgy}.  For each task, we randomly select $15$ curves from each letter class using seed $42$, yielding $45$ curves in total.  The original velocity records are integrated into planar paths, centered, scaled, and resampled to $80$ points.  The baseline applies hierarchical clustering directly to the distance matrix $\delta$.  The contribution-affinity construction improves substantially on \texttt{aho}, \texttt{mwz}, and \texttt{nor}; it remains competitive on \texttt{uvw}, where the direct distance baseline is already strong.  In \texttt{bgy}, depth information does not improve the result.  These results are intended as an exploratory application rather than a comprehensive clustering study.  They indicate that the decomposed CBD and FW-CBD representations may provide complementary information to the curve distance.

\section{Conclusion and future work}
\label{sec:discussion}
The CBD family gives a parameterization-invariant, band-based notion of curve centrality for unparameterized planar curves.
Compared with halfspace-based curve depths, it measures how much of a curve lies inside the region between two others, and so orders curves by band occupancy rather than by halfspace separability.
Empirically, FW-CBD tends to improve discrimination on highly overlapping classes, while length penalization helps when the competing classes differ systematically in arc length.

Two limitations deserve mention.
The present development is essentially planar. While the convex-combination band
extends naturally to $\R^d$, the FW band relies on planar enclosure geometry
and would need to be reformulated for $d\ge 3$.
CBD and its variants also require $O(n^2)$ pairwise band constructions, which limits scalability for large samples.
Higher-dimensional extensions of the band construction and scalable approximations to the pairwise depth computation are therefore the natural next steps.

\appendix
\section{Proofs for Section~\ref{sec:fw_props}}
\label{app:proofs}

This appendix collects the proofs referenced in Section~\ref{sec:fw_props} and the FW half of Lemma~\ref{lemma:band_equivariance}.
\begin{proof}[Proof of Proposition~\ref{prop:fw_welldef} (Well-definedness)]
\emph{Nonemptiness of $\mathcal M(A,B)$.}
For an open $C\in\Cset$, $\mathcal G(C)=\{\beta_C,\beta_C^\leftarrow\}$ is finite.
For a closed $C\in\Cset$, $\mathcal G(C)$ is the image of the compact parameter space $\bbS^1\times\{\pm 1\}$ under the shift-and-orientation map; equipped with the quotient topology induced by that map, $\mathcal G(C)$ is a compact topological space, and the endpoint map $\alpha\mapsto(\alpha(0),\alpha(1))$ is continuous on it.
In all four mixed cases (open/open, open/closed, closed/open, closed/closed), the product $\mathcal G(A)\times\mathcal G(B)$ is compact, and $\Delta:\mathcal G(A)\times\mathcal G(B)\to\R_{\ge 0}$ is continuous (the endpoint maps $(\alpha,\beta)\mapsto(\alpha(0),\alpha(1),\beta(0),\beta(1))$ are continuous in each component).
Hence $\Delta$ attains its minimum, and $\mathcal M(A,B)\neq\varnothing$.
Moreover $\mathcal M(A,B)$ is closed because it is the preimage of the singleton $\{\min\Delta\}$ under the continuous map $\Delta$, and is therefore compact.

\emph{Each $E(\alpha,\beta)$ is closed and Borel.}
By Definition~\ref{def:fwb}, $W(\alpha,\beta)=\gamma_{\alpha,\beta}(\bbS^1)$ is compact.
The mod-$2$ winding number $n_2(\gamma_{\alpha,\beta};\cdot)$ defined in \eqref{eq:n2-degree} is a continuous integer-valued function on each connected component of $\R^2\setminus W$, hence locally constant; reducing mod $2$, $n_2$ is constant on each component.
On the unique unbounded component $U_\infty$ of $\R^2\setminus W$, $n_2$ vanishes.
Indeed, $n_2$ is constant on $U_\infty$, and for $p$ sufficiently far from $W$ the map $(\gamma_{\alpha,\beta}(t)-p)/\|\gamma_{\alpha,\beta}(t)-p\|$ is contained in an arbitrarily small arc of $\bbS^1$, so its degree is $0$.
Write
\[
P_1:=\bigl\{p\in\R^2\setminus W:\,n_2(\gamma_{\alpha,\beta};p)=1\bigr\},
\qquad
E(\alpha,\beta)=W\cup P_1.
\]
$P_1$ is open because it is a union of open connected components of $\R^2\setminus W$.
If $x\in\R^2\setminus(W\cup P_1)$, then $x$ is in a parity-$0$ component of $\R^2\setminus W$, which is open, so $x$ has an open neighborhood disjoint from $W\cup P_1$.
Therefore the complement of $W\cup P_1$ is open, i.e.\ $E(\alpha,\beta)$ is closed and hence Borel.

\emph{Uniform boundedness and compactness of $R_{\mathrm{FW}}$.}
Let $K\subset\R^2$ be a closed Euclidean ball containing $\Gamma_A\cup\Gamma_B$.
Since $K$ is convex, every closure chord between endpoints of admissible representatives lies in $K$, so $W(\alpha,\beta)\subseteq K$ for every admissible pair.
The unbounded component of $\R^2\setminus W(\alpha,\beta)$ then contains $\R^2\setminus K$.
Since $n_2$ vanishes on the unbounded component, all parity-$1$ components are contained in $K$; together with $W(\alpha,\beta)\subseteq K$, this gives $E(\alpha,\beta)\subseteq K$. It follows that
$
\bigcup_{(\alpha,\beta)\in\mathcal M(A,B)}E(\alpha,\beta)\subseteq K,
$
so
$
R_{\mathrm{FW}}(A,B)
=
\cl\!\Bigl(\bigcup_{(\alpha,\beta)\in\mathcal M(A,B)}E(\alpha,\beta)\Bigr)
\subseteq K.
$
Thus $R_{\mathrm{FW}}(A,B)$ is closed by construction and bounded, hence compact and Borel-measurable.
Therefore $P_{\mathrm{FW}}(C\mid A,B)=\mu_C(R_{\mathrm{FW}}(A,B))$ is well-defined for every $C\in\Cset$.
\end{proof}

\begin{proof}[Proof of Proposition~\ref{prop:fw_param} (Reparameterization invariance)]
We show that $\mathcal G(C)$ is intrinsic to $C\in\Cset$; the conclusion for $R_{\mathrm{FW}}(A,B)$ then follows because Definition~\ref{def:normapara} and \eqref{eq:fw-band} use $A$ and $B$ only through $\mathcal G(A)$ and $\mathcal G(B)$.

Fix $C\in\Cset$ and let $\beta_C^{\arc},\widetilde\beta_C^{\arc}$ be two arc-length representatives of $C$.
By Theorem 2.7.6 of \citet{burago2001course} the arc-length representative is unique up to a precise list of admissible reparameterizations:
\begin{itemize}
\item For $C$ open: $\widetilde\beta_C^{\arc}=\beta_C^{\arc}$ or $\widetilde\beta_C^{\arc}(s)=\beta_C^{\arc}(L(C)-s)$.
\item For $C$ closed: $\widetilde\beta_C^{\arc}(s)=\beta_C^{\arc}(s+\theta\bmod L(C))$ or $\widetilde\beta_C^{\arc}(s)=\beta_C^{\arc}(\theta-s\bmod L(C))$ for some $\theta\in[0,L(C))$.
\end{itemize}
The corresponding normalized representatives $\widetilde\beta_C(t)=\widetilde\beta_C^{\arc}(L(C)t)$ therefore lie in $\mathcal G(C)$ as defined from $\beta_C$.
We verify that the family generated by $\widetilde\beta_C$ via Definition~\ref{def:normapara} equals $\mathcal G(C)$.

\emph{Open case.}
If $\widetilde\beta_C=\beta_C$ the families are identical. If $\widetilde\beta_C=\beta_C^\leftarrow$, then $(\widetilde\beta_C)^\leftarrow=(\beta_C^\leftarrow)^\leftarrow=\beta_C$, so the generated family $\{\widetilde\beta_C,(\widetilde\beta_C)^\leftarrow\}=\{\beta_C^\leftarrow,\beta_C\}=\mathcal G(C)$.

\emph{Closed case, shift.}
If $\widetilde\beta_C=\beta_C^\phi$ for some $\phi\in[0,1)$, we apply Definition~\ref{def:normapara} with $\widetilde\beta_C$ in place of $\beta_C$: the shifts are $\widetilde\beta_C^\theta(t):=\widetilde\beta_C(t{+}\theta\bmod 1)$ and their reversals $(\widetilde\beta_C^\theta)^\leftarrow(t):=\widetilde\beta_C(\theta{-}t\bmod 1)$.
Computing:
\[
\widetilde\beta_C^\theta(t)=\beta_C(t{+}\theta{+}\phi\bmod 1)=\beta_C^{\theta+\phi\bmod 1}(t),
\quad
(\widetilde\beta_C^\theta)^\leftarrow(t)=\beta_C(\phi{+}\theta{-}t\bmod 1)=(\beta_C^{\phi+\theta\bmod 1})^\leftarrow(t).
\]
As $\theta$ ranges over $[0,1)$, so does $\theta{+}\phi\bmod 1$, so the shift-family and reversal-family generated by $\widetilde\beta_C$ are precisely $\{\beta_C^\psi:\psi\in[0,1)\}$ and $\{(\beta_C^\psi)^\leftarrow:\psi\in[0,1)\}$ respectively.
Hence $\mathcal G_{\widetilde\beta_C}(C)=\mathcal G_{\beta_C}(C)$.

\emph{Closed case, reversal.}
If $\widetilde\beta_C=(\beta_C^\phi)^\leftarrow$, applying Definition~\ref{def:normapara} to $\widetilde\beta_C$ gives
\[
\widetilde\beta_C^\theta(t)=\beta_C(\phi{-}\theta{-}t\bmod 1)=(\beta_C^{\phi-\theta\bmod 1})^\leftarrow(t),
\quad
(\widetilde\beta_C^\theta)^\leftarrow(t)=\beta_C^{\phi-\theta\bmod 1}(t).
\]
Note that when $\widetilde\beta_C$ is itself a reversal, shifts and reversals swap roles in the generated family: shifts of $\widetilde\beta_C$ land in the reversal component of $\mathcal G(C)$, and reversals of $\widetilde\beta_C$ land in the shift component.
As $\theta$ ranges over $[0,1)$, so does $\phi{-}\theta\bmod 1$, so both components of $\mathcal G_{\widetilde\beta_C}(C)$ match those of $\mathcal G_{\beta_C}(C)$.
Hence $\mathcal G_{\widetilde\beta_C}(C)=\mathcal G_{\beta_C}(C)$.

Hence $\mathcal G(C)$ is intrinsic to $C\in\Cset$.
Consequently $R_{\mathrm{FW}}(A,B)$ is intrinsic to $(A,B)\in\Cset\times\Cset$, since it is defined through $\mathcal G(A)$, $\mathcal G(B)$, and the closure operation, all of which depend only on the equivalence classes $A$ and $B$.
\end{proof}

\begin{proof}[Proof of Proposition~\ref{prop:fw_sym} (Symmetry)]
Fix any admissible pair $(\alpha,\beta)\in\mathcal G(A)\times\mathcal G(B)$.
Since line segments are symmetric as sets, $[x,y]=[y,x]$, the boundary trace satisfies $W(\alpha,\beta)=W(\beta,\alpha)$.

For the mod-$2$ winding interior, reversing the orientation of the closed walk does not change its mod-$2$ winding number around any point.
Swapping $(\alpha,\beta)\mapsto(\beta,\alpha)$ in Definition~\ref{def:fwb} yields the loop $\beta\ast[\beta(1),\alpha(1)]\ast\alpha^\leftarrow\ast[\alpha(0),\beta(0)]$, which is obtained from $\gamma_{\alpha,\beta}$ by a cyclic re-parameterization composed with overall reversal; both operations preserve mod-$2$ winding number around every point of $\R^2\setminus W$.
Hence
\[
\{p\notin W(\alpha,\beta):\,n_2(\gamma_{\alpha,\beta};p)=1\}
=\{p\notin W(\beta,\alpha):\,n_2(\gamma_{\beta,\alpha};p)=1\},
\]
and $E(\alpha,\beta)=E(\beta,\alpha)$.

The closure score is also symmetric, $\Delta(\alpha,\beta)=\Delta(\beta,\alpha)$, so the swap $(\alpha,\beta)\mapsto(\beta,\alpha)$ defines a bijection $\mathcal M(A,B)\leftrightarrow\mathcal M(B,A)$.
Therefore $\bigcup_{\mathcal M(A,B)} E(\alpha,\beta)=\bigcup_{\mathcal M(B,A)} E(\beta,\alpha)$, and taking closures gives $R_{\mathrm{FW}}(A,B)=R_{\mathrm{FW}}(B,A)$.
\end{proof}

\begin{proof}[Proof of Proposition~\ref{prop:fw_subset_global} (FW band is contained in the global band)]
Fix $\alpha\in\mathcal{G}(A)$ and $\beta\in\mathcal{G}(B)$, and let
$\sigma_1$ and $\sigma_0$ parameterize the closure chords
$[\alpha(1),\beta(1)]$ from $\alpha(1)$ to $\beta(1)$ and
$[\beta(0),\alpha(0)]$ from $\beta(0)$ to $\alpha(0)$, respectively.
Let $\delta$ parameterize $[\beta(1),\alpha(0)]$ from $\beta(1)$ to
$\alpha(0)$, and define
\[
\ell_A:=\alpha*\sigma_1*\delta,
\qquad
\ell_B:=\delta^{-1}*\beta^{\leftarrow}*\sigma_0.
\]
Consider
\[
K_A:=\bigcup_{x\in\Gamma_A}[x,\beta(1)],
\qquad
K_B:=\bigcup_{y\in\Gamma_B}[\alpha(0),y].
\]
By the definition of the global band,
\[
K_A,K_B\subseteq\mathcal{R}(A,B).
\]
Moreover, $K_A$ is star-shaped with respect to $\beta(1)$ and
$K_B$ with respect to $\alpha(0)$, while
$\ell_A(S^1)\subseteq K_A$ and $\ell_B(S^1)\subseteq K_B$.
Now let $p\notin\mathcal{R}(A,B)$. Since $K_A$ and $K_B$ are star-shaped and avoid $p$, the loops $\ell_A$ and $\ell_B$ can each be continuously contracted to a point without passing through $p$. By homotopy invariance of the mod-$2$ winding number,
\[
n_2(\ell_A;p)=n_2(\ell_B;p)=0.
\]
Furthermore,
\[
\ell_A*\ell_B
=
\alpha*\sigma_1*\delta*\delta^{-1}*
\beta^{\leftarrow}*\sigma_0.
\]
Since $\delta\subseteq\mathcal{R}(A,B)$, the backtracking path
$\delta*\delta^{-1}$ is null-homotopic in
$\mathbb{R}^2\setminus\{p\}$. Thus $\ell_A*\ell_B$ is homotopic to $\gamma_{\alpha,\beta}$. By homotopy invariance and additivity of the mod-$2$ winding number,
\[
n_2(\gamma_{\alpha,\beta};p)
=
n_2(\ell_A;p)+n_2(\ell_B;p)
=
0.
\]
Therefore
$
\{p\in\mathbb{R}^2\setminus W(\alpha,\beta):
n_2(\gamma_{\alpha,\beta};p)=1\}
\subseteq\mathcal{R}(A,B).
$ Since also
$W(\alpha,\beta)\subseteq\mathcal{R}(A,B)$, Definition~\ref{def:fwb} gives
\[
E(\alpha,\beta)\subseteq\mathcal{R}(A,B).
\]
This holds for every admissible pair, and hence
\[
\bigcup_{(\alpha,\beta)\in\mathcal{M}(A,B)}
E(\alpha,\beta)
\subseteq\mathcal{R}(A,B).
\]
Finally, $\mathcal{R}(A,B)$ is compact, being the continuous image of
$\Gamma_A\times\Gamma_B\times[0,1]$ under the convex-combination map,
and is therefore closed. Taking closures yields
$
\mathcal{R}_{\mathrm{FW}}(A,B)\subseteq\mathcal{R}(A,B).
$
\end{proof}

\begin{proof}[Proof of the FW half of Lemma~\ref{lemma:band_equivariance}]
Let $g(x)=rQx+b$ with $r>0$ and $Q$ orthogonal.
We first verify that $g$ induces a bijection
$\mathcal G(A)\to\mathcal G(g(A))$ for every $A\in\Cset$.
If $\alpha\in\mathcal G(A)$ is a constant-speed representative of $A$ with speed
$L(A)$, then $g\circ\alpha$ is continuous with
\[
|(g\circ\alpha)'(t)|
=
r\,|\alpha'(t)|
=
r\,L(A)
=
L(g(A))
\]
almost everywhere, since a similarity scales Euclidean speeds and arc length by
$r$.
By the arc-length uniqueness described in Section~\ref{sec:curve_space},
constant-speed representatives of $g(A)$ with speed $L(g(A))$ are unique up to
orientation reversal in the open case and up to circular shift and reversal in
the closed case. Hence $g\circ\alpha\in\mathcal G(g(A))$.
Conversely, for any $\widetilde\alpha\in\mathcal G(g(A))$,
$g^{-1}\circ\widetilde\alpha\in\mathcal G(A)$ by the same argument applied to
$g^{-1}$.
The map $\alpha\mapsto g\circ\alpha$ is therefore a bijection
$\mathcal G(A)\to\mathcal G(g(A))$, and likewise for $B$.
Thus the product map
\[
T_g:(\alpha,\beta)\mapsto(g\circ\alpha,g\circ\beta)
\]
is a bijection from $\mathcal G(A)\times\mathcal G(B)$ to
$\mathcal G(g(A))\times\mathcal G(g(B))$.

We next prove that, for every admissible pair $(\alpha,\beta)$,
\[
g(E(\alpha,\beta))=E(g\circ\alpha,g\circ\beta).
\]
First suppose that both $A$ and $B$ are closed.
By Lemma~\ref{lemma:closed_parity_decomp},
\[
E(\alpha,\beta)
=
\widehat E_2(A,B)\cup[\alpha(0),\beta(0)].
\]
We claim that
\[
g(\widehat E_2(A,B))=\widehat E_2(g(A),g(B)).
\]
Indeed, $g(\Gamma_A)=\Gamma_{g(A)}$ and $g(\Gamma_B)=\Gamma_{g(B)}$.
Moreover, for any closed curve $C$ and any $p\in\R^2\setminus\Gamma_C$,
the same mod-$2$ degree argument gives
\[
n_2(g(C);g(p))=n_2(C;p).
\]
To see this, if $\eta\in\mathcal G(C)$ and
\[
\varphi_p(t)
:=
\frac{\eta(t)-p}{\|\eta(t)-p\|},
\]
then the corresponding normalized direction map for $g\circ\eta$ around
$g(p)$ is $Q\varphi_p(t)$.
Composition with $Q\in O(2)$ multiplies the integer degree by
$\det Q\in\{+1,-1\}$, and this sign disappears modulo $2$.
Hence $g(I_2(C))=I_2(g(C))$.
Since $g$ is bijective and preserves symmetric difference,
\[
g\bigl(I_2(A)\triangle I_2(B)\bigr)
=
I_2(g(A))\triangle I_2(g(B)).
\]
Thus $g(\widehat E_2(A,B))=\widehat E_2(g(A),g(B))$.
Also,
\[
g([\alpha(0),\beta(0)])
=
[g(\alpha(0)),g(\beta(0))].
\]
Applying Lemma~\ref{lemma:closed_parity_decomp} to
$g\circ\alpha\in\mathcal G(g(A))$ and
$g\circ\beta\in\mathcal G(g(B))$ gives
\[
g(E(\alpha,\beta))=E(g\circ\alpha,g\circ\beta).
\]

It remains to consider the cases in which at least one of $A$ and $B$ is open.
For every admissible pair $(\alpha,\beta)$, $g$ sends curves to curves and
segments to segments, so
\[
g\bigl(W(\alpha,\beta)\bigr)=W(g\circ\alpha,g\circ\beta),
\qquad
g\circ\gamma_{\alpha,\beta}
=
\gamma_{g\circ\alpha,g\circ\beta}.
\]
We verify directly that the mod-$2$ degree is preserved.
Fix $p\in\R^2\setminus W(\alpha,\beta)$.
Since
\[
g\bigl(W(\alpha,\beta)\bigr)=W(g\circ\alpha,g\circ\beta),
\]
we have $g(p)\notin W(g\circ\alpha,g\circ\beta)$.
Let
\[
\varphi_p(t)
:=
\frac{\gamma_{\alpha,\beta}(t)-p}
{\|\gamma_{\alpha,\beta}(t)-p\|},
\qquad
\varphi'_{g(p)}(t)
:=
\frac{(g\circ\gamma_{\alpha,\beta})(t)-g(p)}
{\|(g\circ\gamma_{\alpha,\beta})(t)-g(p)\|}.
\]
Since $g(y)-g(p)=rQ(y-p)$ and
$\|g(y)-g(p)\|=r\|y-p\|$, we have
\[
\varphi'_{g(p)}(t)=Q\varphi_p(t).
\]
Left composition with $Q\in O(2)$ multiplies the integer Brouwer degree by
$\det Q\in\{+1,-1\}$.
Since $-1\equiv 1\pmod 2$, its reduction modulo $2$ is unchanged.
Therefore
\[
n_2(\gamma_{g\circ\alpha,g\circ\beta};g(p))
=
\deg(\varphi'_{g(p)})\bmod 2
=
\deg(\varphi_p)\bmod 2
=
n_2(\gamma_{\alpha,\beta};p).
\]
Since $g$ is bijective and maps $W(\alpha,\beta)$ onto
$W(g\circ\alpha,g\circ\beta)$, it maps the parity-$1$ subset of
$\R^2\setminus W(\alpha,\beta)$ bijectively onto the parity-$1$ subset of
$\R^2\setminus W(g\circ\alpha,g\circ\beta)$.
Hence again
\[
g(E(\alpha,\beta))=E(g\circ\alpha,g\circ\beta).
\]
For the closure score,
$
\Delta(g\circ\alpha,g\circ\beta)
=
r\,\Delta(\alpha,\beta).
$
Combined with the bijection above, this implies that the product map
$T_g$ sends $\mathcal M(A,B)$ bijectively onto
$\mathcal M(g(A),g(B))$.
Taking the union over $\mathcal M(A,B)$ and applying $g$ pointwise gives
\[
g\!\Bigl(\bigcup_{(\alpha,\beta)\in\mathcal M(A,B)}E(\alpha,\beta)\Bigr)
=
\bigcup_{(\alpha,\beta)\in\mathcal M(A,B)}
E(g\circ\alpha,g\circ\beta)
=
\bigcup_{(\alpha',\beta')\in\mathcal M(g(A),g(B))}
E(\alpha',\beta').
\]
Since $g$ is a homeomorphism, $g(\cl S)=\cl(g(S))$ for every set
$S\subseteq\R^2$.
Applying this to
\[
S=\bigcup_{(\alpha,\beta)\in\mathcal M(A,B)}E(\alpha,\beta)
\]
yields
\[
g\bigl(R_{\mathrm{FW}}(A,B)\bigr)
=
g\!\Bigl(\cl\!\Bigl(\bigcup_{\mathcal M(A,B)}E(\alpha,\beta)\Bigr)\Bigr)
=
\cl\!\Bigl(\bigcup_{\mathcal M(g(A),g(B))}E(\alpha',\beta')\Bigr)
=
R_{\mathrm{FW}}(g(A),g(B)).
\]
\end{proof}

\bibliographystyle{chicago}
\bibliography{references}

\end{document}